\documentclass[10pt,a4paper]{article}
\pdfoutput=1
\usepackage{graphics,epsfig,graphicx,color}
\graphicspath{{./Figures/}}
\usepackage{subfigure}

\usepackage{amssymb,latexsym}

\usepackage{setspace}

\usepackage{amsmath}

\usepackage{mathrsfs,amsfonts}

\usepackage{amsthm,amsxtra}

\usepackage[final]{showkeys}

\usepackage{stmaryrd}

\usepackage{algorithm}
\usepackage{algorithmicx}
\usepackage{algpseudocode}
\usepackage{mathtools}

\newif\ifPDF
\ifx\pdfoutput\undefined
\PDFfalse
\else
\ifnum\pdfoutput > 0
\PDFtrue
\else
\PDFfalse
\fi
\fi

\ifPDF
\usepackage{pdftricks}
\begin{psinputs}
	\usepackage{pstricks}
	\usepackage{pstcol}
	\usepackage{pst-plot}
	\usepackage{pst-tree}
	\usepackage{pst-eps}
	\usepackage{multido}
	\usepackage{pst-node}
	\usepackage{pst-eps}
\end{psinputs}
\else
\usepackage{pstricks}
\fi

\ifPDF
\usepackage[debug,pdftex,colorlinks=true, 
linkcolor=blue, bookmarksopen=false,
plainpages=false,pdfpagelabels]{hyperref}
\else
\usepackage[dvips]{hyperref}
\fi

\usepackage[openbib]{currvita}

\usepackage{fancyhdr}

\newtheorem{theorem}{Theorem}[section]
\newtheorem{lemma}[theorem]{Lemma}

\newtheorem{remark}[theorem]{Remark}

\newcommand{\eps}{\varepsilon}

\newcommand{\bbC}{\mathbb C} 
 \newcommand{\bbN}{\mathbb N}
 
\newcommand{\bbR}{\mathbb R} \newcommand{\bbS}{\mathbb S}

 \newcommand{\bn}{\mathbf n}

 \newcommand{\bv}{\mathbf v} 
 \newcommand{\bx}{\mathbf x} 
\newcommand{\by}{\mathbf y}

\newcommand{\cC}{\mathcal C} \newcommand{\cD}{\mathcal D}

\newcommand{\cI}{\mathcal I} \newcommand{\cJ}{\mathcal J}
\newcommand{\cK}{\mathcal K} \newcommand{\cL}{\mathcal L}
 
\newcommand{\cO}{\mathcal O}  
\newcommand{\cQ}{\mathcal Q} 
 
\newcommand{\cU}{\mathcal U}

\newenvironment{keywords}
{\noindent{\bf Key words. }\small}{\par\vspace{1ex}}
\newenvironment{AMS}
{\noindent{\bf AMS subject classifications 2000.}\small}{\par}

\makeatletter
\newcommand{\chapterauthor}[1]{%
	{\parindent0pt\vspace*{-25pt}%
		\linespread{1.1}\large\scshape#1%
		\par\nobreak\vspace*{35pt}}
	\@afterheading%
}
\makeatother

\title{Separability of the kernel function in an integral formulation for anisotropic radiative transfer equation}
\author{
	Kui Ren\thanks{
        Department of Applied Physics and Applied Mathematics, Columbia University, New York, NY 10027; kr2002@columbia.edu}
    \and Hongkai Zhao\thanks{
        Department of Mathematics, University of California, Irvine, CA 92697; zhao@uci.edu}
    \and     Yimin Zhong\thanks{
        Department of Mathematics, University of California, Irvine, CA 92697; yiminz@uci.edu}    }
\begin{document}
\maketitle

\begin{abstract}
We study in this work an integral formulation for the radiative transfer equation (RTE) in anisotropic media with truncated approximation to the scattering phase function. The integral formulation consists of a coupled system of integral equations for the angular moments of the transport solution. We analyze the approximate separability of the kernel functions in these integral formulations, deriving asymptotic lower and upper bounds on the number of terms needed in a separable approximation of the kernel functions as the moment grows. Our analysis provides the mathematical understanding on when low-rank approximations to the discretized integral kernels can be used to develop fast numerical algorithms for the corresponding system of integral equations.
\end{abstract}

\section{Introduction}
\label{SEC:INTRO}

The radiative transfer equation (RTE) is an important mathematical model for the quantitative description of particle transport processes in many physical and biological systems~\cite{CeBaBeAi-TTSP99,Chandrasekhar-Book60,HeGr-AJ41,Larsen-NSE88,LeMi-Book93,Mokhtar-Book97,SpKuCh-JQSRT01,TuFrDuKl-JCP04}. In recent years, research interests in the RTE has been fueled with its newly-dicovered application in emerging areas such as optical imaging~\cite{Arridge-IP99,Bal-IP09,DeThUr-JCP12,DiRe-JCP14,GaZh-TTSP09,GaZh-OE10,GoYa-SIAM16,KiMo-IP06,LaLiUh-SIAM18,LiSu-arXiv19,MaRe-CMS14,Ren-CiCP10,ReBaHi-SIAM06,SaTaCoAr-IP13,Tamasan-IP02,Wang-AIHP99,ZhZh-SIAM18} and imaging in random media~\cite{BaCaLiRe-IP07,BaRe-SIAM08,BoGa-PRE16}.

In the steady-state, RTE is usually formulated as the following integro-differential equation:
\begin{equation}\label{eq:rte1}
\begin{aligned}
\bv \cdot \nabla u(\bx, \bv) + \sigma_t(\bx) u(\bx, \bv) &= \sigma_s(\bx)\int_{\bbS^{d-1}} p(\bv, \bv') u(\bx, \bv') d\bv' + q(\bx, \bv),&\quad& \text{ in }D
\\
u(\bx, \bv) &= f(\bx,\bv), &\quad &\text{ on }\Gamma_{-}
\end{aligned}
\end{equation}
where $u(\bx, \bv)$ is density of the radiative particles at location $\bx\in\Omega$ traveling in the direction $\bv\in\bbS^{d-1}$. The physical space $\Omega$ is assumed to be a bounded convex set in $\bbR^d$ and the angular space, that is the space of all possible traveling directions, $\bbS^{d-1}$, is the unit sphere in $\bbR^{d}$. The phase space is defined as $D =\Omega \times \bbS^{d-1}$ with incoming and outgoing boundaries, $\Gamma_{-}$ and $\Gamma_{+}$ respectively, given as $\Gamma_{\pm}=\{(\bx,\bv)\in \partial\Omega\times \bbS^{d-1}\,|\, \pm\bv\cdot \bn(\bx) > 0\}$, $\bn(\bx)$ being the outward normal vector at $\bx \in \partial\Omega$. The functions $q$ and $f$ denote respectively the internal and boundary sources of particles in the problem.

The coefficients $\sigma_t(\bx)$ and $\sigma_s(\bx)$ are the transport and scattering coefficients respectively. For the well-posedness of solution, we assume that there exist positive constants $k_0$, $\sigma_0$, and $\sigma_1$ such that
\begin{equation}\nonumber
\begin{aligned}
&\sup_{\bx\in\Omega}\frac{\sigma_s(\bx)}{\sigma_t(\bx)} = k_0 < 1 ,\quad \sigma_0 \le \sigma_s(\bx) < \sigma_t(\bx) \le \sigma_1 < \infty.
\end{aligned}
\end{equation}

The scattering phase function $p(\bv, \bv')$ represents the probability of particles with propagation direction $\bv$ being scattered into direction $\bv'$. A very common choice for $p(\bv, \bv')$ in application is the Henyey-Greenstein function which depends on $\bv$ and $\bv'$ only through their inner product $\bv\cdot\bv':=\cos\theta$:
\begin{equation}\label{EQ:HG}
p(\bv, \bv')=p(\bv\cdot\bv')=p_{HG}(\cos\theta): = \begin{dcases}
\frac{1-g^2}{(1 + g^2 - 2g \cos\theta)}\quad &d = 2, \\
\frac{1-g^2}{(1+g^2 - 2g \cos\theta)^{3/2}}\quad &d = 3,
\end{dcases}
\end{equation}
where $g\in(-1,1)$ is the anisotropy parameter. For the simplicity of the presentation, we have normalized the surface measure $d\bv'$ over $\bbS^{d-1}$ so that the scattering term of Equation~\eqref{eq:rte1} does not carry the factor $1/|\bbS^{d-1}|$ in front of the integral. For a scattering phase function that can be parameterized as in~\eqref{EQ:HG}, the normalization of the surface measure, $d\bv=d\theta/2\pi$ when $d=2$ and $d\bv=\sin\theta d\theta d\varphi/4\pi$ when $d=3$ ($\varphi$ being the azimuthal angle), leads to the usual normalization conditions on the scattering phase function: $\int_{\bbS^{d-1}} p(\bv, \bv') d\bv =\int_{\bbS^{d-1}} p(\bv, \bv') d\bv' =1$, required for the scattering process to have mass conservation.

The major challenge of solving the RTE~\eqref{eq:rte1} is due to its high dimensionality: the equation is posed in the phase space $D = \Omega\times \bbS^{d-1}$ which has dimension $2d-1$ when the physical space $\Omega$ is in $\bbR^d$ ($d\ge 2$). Therefore, dimension reduction, or model reduction in general, is often preferred in the study of RTE. A classical dimension reduction method is the diffusion approximation. This is the case when the mean free path of the particles is very small (assuming that the size of the domain $\Omega$ is of order $1$). In this case, one can show that the solution $u(\bx, \bv)$ of the RTE becomes independent of the directional variable $\bv$ when the mean free path goes to zero, and converges to the solution to the classical diffusion equation. Therefore in the diffusion limit, the dimension of the problem reduces to the dimension of the physical space. In~\cite{ReZhZh-arXiv19}, a different dimension reduction method is introduced for the RTE in the case of isotropic scattering, that is the case of $p(\bv, \bv')\equiv 1$. In this case, one can derive an integral equation for the average of $u$ over the direction variable $\bv$, that is, the quantity $\cU(\bx)= \int_{\bbS^{d-1}} u(\bx, \bv) d\bv$. More precisely, assuming that the source functions $q(\bx, \bv) = q(\bx)$ and $f(\bx, \bv)=0$ for simplicity, then $\cU(\bx)$ satisfies the following Fredholm integral equation:
\begin{equation}\label{EQ:RTE Isot Integ}
	\cU(\bx) = \frac{1}{\nu_d}\int_{\Omega} \frac{E(\bx, \by)}{|\bx - \by|^{d-1}}(\sigma_s(\by)\cU(\by) + q(\by))\mathrm{d}\by\,,
\end{equation}
with the function
\begin{equation}
	E(\bx, \by) = \exp\left( - |\bx - \by|\int_{0}^1 \sigma_t(\bx + (\by-\bx)s) ds\right)
\end{equation}
representing the the path integral of $\sigma_t$ on the segment connecting $\bx$ and $\by$. This integral formulation can be used to solve for $\cU$ with existing fast algorithms for integral equations; see~\cite{ReZhZh-arXiv19} for an algorithm based on the fast multipole method (FMM).

The above integral formulation can be generalized to the case of anisotropic scattering when the scattering kernel $p(\bv, \bv')$ is highly separable, in appropriate sense that we will specify later. This was done in~\cite{FaAnYi-JCP19}. Instead of solving one integral equation, the anisotropic case involves a system of coupled integral equations. Unless the kernel $p$ has only a finite number of modes, one has to truncate the coupled system to obtain a finite system of integral equations for the (generalized) modes of $u$; see~\cite{FaAnYi-JCP19} for more details. The key difference between the integral formulation and the classical ${\rm P_N}$ method (i.e. the method of spherical harmonics truncated at order $N-1$)~\cite{DaLi-Book93-6} is that the truncation in the integral formulation is taken in the scattering phase function while the truncation in the ${\rm P_N}$  method is taken on the modes of the RTE solution $u$. When the scattering kernel $p$ is highly separable, the former approach, that is the integral equation approach, leads to small truncated system that would give highly accurate approximation to the original solution $u$, regardless of the scattering strength of the medium, that is the size of $\sigma_s$.

In this work, we consider the general expansion of the scattering phase function of the form $p(\bv, \bv')=p(\bv\cdot\bv')\equiv p(\cos\theta)$, the Henney-Greenstein function~\eqref{EQ:HG} being a special example, as follows
\begin{equation}\label{EQ:PHASE}
\begin{aligned}
{p}(\cos\theta ) = \begin{dcases}
\sum_{n= -\infty}^{\infty}\chi_n \cos n\theta, &d= 2\\
\sum_{n= 0}^{\infty}(2n+1)\chi_n P_n(\cos\theta), &d= 3
\end{dcases}
\end{aligned}
\end{equation}
where $P_n$ is the $n$th Legendre polynomial and $\chi_n$ is a real number for each $n$. Due to symmetry of cosine function, we may assume $\chi_n = \chi_{-n} = \chi_{|n|}$ for $d=2$. The system of integral equations resulted from such expansions, of the form~\eqref{eq:rte1}, has integral kernels that are quite different from existing integral kernels in the literature (most of which are related to the fundamental solutions to the Laplace operator, the Helmholtz operator and alike)~\cite{EnZh-CPAM18, Hackbusch-Book17}. Our objective in this work is to characterize the separability properties of these kernel functions for the general anisotropic scattering phase function $p$ given in~\eqref{EQ:PHASE}. Since the separability property of an integral kernel is directly related to the numerical rank of the discretized integral operators, our objective is essentially to understand whether or not (hierarchical) low-rank approximations to the corresponding integral kernels exists for such anisotropic scattering phase functions. Existence of such (hierarchical) low-rank approximations is crucial in developing fast solvers for the system of integral equations.

The rest of the paper is organized as follows. In Section~\ref{SEC:FORM}, we introduce the integral formulations of the RTE in two- and three-dimensional domains with with the above scattering phase function. We review also some basic properties of the resulted system of integral equations. We then derive lower and upper bounds for the separability of the integral kernels in the integral formulations in Sections~\ref{SEC:SEP}. Concluding remarks are made in Section~\ref{SEC:CON}.

\section{Truncated integral formulation}
\label{SEC:FORM}

We now present the integral formulation of the radiative transfer equation~\eqref{eq:rte1} with anisotropic scattering. We consider the following $M$-term truncated approximation to the phase scattering function~\eqref{EQ:PHASE}:
\begin{equation}\label{EQ:TRUNC}
\begin{aligned}
{p}(\cos\theta ) \simeq p_M(\cos\theta):= \begin{dcases}
\sum_{n = 0}^{M-1} (2-\delta_{0n})\chi_n \cos n\theta\quad &d= 2,\\
\sum_{n = 0}^{M-1} (2n+1)\chi_n P_n(\cos\theta)\quad &d= 3.
\end{dcases}
\end{aligned}
\end{equation}
where $\delta_{0n}$ denotes Kronecker delta. By the orthogonality of the Fourier basis (resp. Legendre polynomials), the above expansions are the best $M$-term approximations in $L^2(\bbS^{d-1})$.

Some of the most famous examples of truncated approximations of the form~\eqref{EQ:TRUNC} are Chandrasekhar's  one-term and two-term models $p_1=1+\chi_1 \cos\theta$ and $p_2=1+\chi_1 \cos\theta+\chi_2 P_2(\cos\theta)$, and the Rayleigh phase function $p_{Rayleigh}=\dfrac{3}{4}(1+\cos^2\theta)$ (which is the special case of $p_2$ with $\chi_1=0$ and $\chi_2=1/2$)~\cite{Chandrasekhar-Book60}.

\begin{remark}
When the scattering phase function $p(\cos\theta)$ is analytic in $\theta$, one concludes from the Paley-Wiener theorem (for example Theorem 13.2.2 in~\cite{Davis-Book75} for $d=3$) that the coefficients of the expansion, $\chi_n$, decay exponentially with $n$.
\end{remark}

\begin{remark}
For the popular Henyey-Greenstein scattering phase function~\eqref{EQ:HG}, we verify easily that $\chi_n = g^{|n|}$,  $-1\le g\le 1$. When the anisotropy factor $g$ is close to $1$, that is, when the scattering is highly forward-peaking, the above expansion converges very slowly. One therefore needs a large $M$ in~\eqref{EQ:TRUNC} to have good approximation accuracy. There are some other approximation schemes, mostly empirical, in the literature to handle such forward-peaking phase functions. For instance, the $\delta-M$ method tries to approximate highly forward-peaking phase functions with the superposition of a Dirac function $\delta(\theta-\theta_0)$ in the forward direction $\theta_0$ and a $M$-term expansion of the smooth part of the phase function~\cite{Wiscombe-JAS77}. The Dirac function term in the phase function can be integrated out in the scattering operator, resulting in only a modification of the transport coefficient $\sigma_t$ in the same radiative transfer equation~\eqref{eq:rte1}. Therefore, the $\delta-M$ approximation can be treated in exactly the same way as the general $M$-term truncation~\eqref{EQ:TRUNC}.
\end{remark}

\subsection{The two-dimensional case}

When $d=2$, we use the parameterization $\bv:= (\cos\theta, \sin\theta)$, normalized measure $d\bv = \frac{d\theta}{2\pi}$ and the notation $u(\bx, \theta):=u(\bx, \bv)$. Taking the $M$-term approximation of the scattering phase function, note that $\cos n\theta = \frac{1}{2}(e^{in\theta} + e^{-in\theta})$, we can write the RTE~\eqref{eq:rte1} as
\begin{equation}\label{eq:RTE3}
\bv \cdot \nabla u(\bx, \theta) +{\sigma}_t(\bx) u(\bx,\theta) = {\sigma}_s(\bx) \sum_{n\in \cJ}\frac{1}{2\pi}\int_{\bbS} \chi_{|n|} e^{in(\theta-\theta')} u(\bx,\theta')d\theta' + q(\bx,\theta)\,,
\end{equation}
where the index set $\cJ = \{n\in\bbN : |n|< M \}$. 
Let us define the $n$th angular Fourier modes of solution $u(\bx,\theta)$ and source $q(\bx, \theta)$ as follows
\begin{equation}\nonumber
u_n(\bx) = \frac{1}{2\pi}\int_{\bbS}  u(\bx,\theta)e^{-in\theta} d\theta,\quad q_n(\bx) = \frac{1}{2\pi}\int_{\bbS} q(\bx, \theta) e^{-in\theta} d\theta\,.
\end{equation}
We then check that if the angular Fourier modes of $q(\bx, \theta)$ is only nonzero when $n\in \cJ$, that is, $q_n=0$ $\forall n\notin \cJ$, then the equation~\eqref{eq:RTE3} could be rewritten as
\begin{equation}\label{EQ:RTE4}
	\bv \cdot \nabla u(\bx, \theta) + {\sigma}_t(\bx) u(\bx,\theta) = {\sigma}_s(\bx) \sum_{n\in\cJ} \chi_{|n|} e^{in \theta} u_n(\bx) + \sum_{n\in\cJ}  e^{in \theta}q_n(\bx)\,.
\end{equation}
Following the same derivation as in~\cite{ReZhZh-arXiv19}, we first integrate the transport part, that is the left side, of the equation by the method of characteristics to obtain that
\begin{equation}\label{EQ:U}
	u(\bx, \theta) = \int_0^{\tau_{-}(\bx, \theta)} \exp\left(-\int_0^l {\sigma}_t(\bx - s\bv) ds\right) \phi(\bx - l\bv, \theta)dl\,,
\end{equation}
where $\tau_{-}(\bx, \theta) = \sup\{s\,|\,\bx -s\bv \in \Omega\}$ is travel distance inside the physical domain from point $\bx$ along direction $-\bv$ (i.e. $\theta+\pi$) to reach the domain boundary, and $\phi(\bx, \theta)$ denotes the right-hand-side of~\eqref{EQ:RTE4}: 
\begin{equation}
	\phi(\bx, \theta) = {\sigma}_s(\bx) \sum_{n\in\cJ} \chi_{|n|} e^{in \theta} u_n(\bx) + \sum_{n\in\cJ}  e^{in \theta}q_n(\bx)\,.
\end{equation}
We then compute the $k$th moment of $u(\bx, \theta)$ by multiplying~\eqref{EQ:U} with $\frac{1}{2\pi} e^{-ik\theta}$ and integrate over $[0, 2\pi]$. The result, after a change of variable from Cartesian to polar coordinate, is the following system of coupled integral equations for the angular Fourier modes of $u(\bx, \bv)$, $\{u_k(\bx)\}_{k\in\cJ}$:
\begin{equation}\label{EQ:Integ Form k}
	u_k(\bx) = \frac{1}{2\pi}\int_{D} \frac{E(\bx , \by)}{|\bx - \by|}\left({\sigma}_s(\by)\sum_{n\in\cJ}\chi_{|n|} e^{-i(k-n)\theta} u_n(\by) +\sum_{n\in\cJ}  e^{-i(k-n)\theta}q_n(\by) \right)d\by\,,\ k\in\cJ
\end{equation}
where $\theta = \arg(\bx - \by)$ and $E(\bx,\by) = \exp\left(-\int_0^{|\bx-\by|} {\sigma}_t\left(\bx - s\bv\right) ds\right)$ is the total attenuation from $\bx$ to $\by$.  

The system of integral equations in~\eqref{EQ:Integ Form k} can be written in a more compact form. To do that, let us define the integral operator $\mathcal{K}_n: L^2(\Omega)\to L^2(\Omega)$:
\begin{equation}\label{EQ:K2}
\begin{aligned}
&\mathcal{K}_n f := \frac{1}{2\pi}\int_{\Omega} G_n(\bx, \by)   f(\by) d\by,\;\text{where}\;\; G_n(\bx, \by) = \frac{E(\bx,\by)}{|\bx- \by|}e^{-in\theta}
\end{aligned}
\end{equation}
and the vector space $V$:
\begin{equation}\nonumber
V = \{v_{\cJ} := (v_j)_{j\in\cJ}\,|\,\forall j\in\cJ , v_j\in L^2(\Omega;\bbC)\text{ and }\overline{v}_j = v_{-j}  \}
\end{equation}
equipped with induced norm  
$
\|v_{\cJ}\|_{V}^2 = \sum_{j\in\cJ} \|v_j\|^2_{L^2(\Omega)}$, where $\overline{v}_j$ is the complex conjugate of $v_j$.
Then the function vector $\mathcal{U} = (u_{j})_{j\in\cJ}\in V$ and satisfy a system of integral equations
\begin{equation}\label{eq:sys}
(\mathcal{I} - \mathcal{L} \mathcal{D}) \mathcal{U} = \mathcal{L}\mathcal{Q}\,,
\end{equation}
where $\cI$ is the identity operator, $\mathcal{L}$ is an operator Toeplitz matrix with entry $\mathcal{L}_{ij} = \mathcal{K}_{i-j}$ and $\mathcal{D}$ is a multiplicative operator matrix with only diagonal entries $\mathcal{D}_{nn} = \chi_{|n|} {\sigma}_s$, $\forall n\in\cJ$. The source term $\mathcal{Q} = (q_j)_{j\in\cJ}\in V$.

\subsection{The three-dimensional case}

The derivation in the three-dimensional case is similar. First, we observe that by the addition theorem, the Legendre polynomial satisfies the following relation
\begin{equation}
	P_n(\bv\cdot \bv') =  \frac{4\pi}{2n+1} \sum_{m=-n}^n Y_{nm}(\bv) Y^{\ast}_{nm}(\bv') \text{ with }\bv,\bv'\in\bbS^2\,,
\end{equation}
where $Y_{nm}$ is the $m$th spherical harmonics of degree $n$. We then define the spherical harmonic moment of $u$ and $q$ in a similar way as in the two-dimensional case:
\begin{equation}\nonumber
u_{nm}(\bx) =\int_{\bbS^2}  u(\bx,\bv) Y^{\ast}_{nm}(\bv) d\bv,\quad q_{nm}(\bx) = \int_{\bbS^2} q(\bx, \bv) Y_{nm}^{\ast}(\bv)d\bv\,,
\end{equation}
where we emphasize that the surface measure $d\bv$ is normalized.

Under the same assumption that the source function $q$ has only nonzero moment within the index set $\cJ'= \{(n,m)\in\bbN^2: 0\le n < M, |m|\le n\}$, we can write the RTE again as
\begin{equation}
	\bv\cdot \nabla u(\bx, \bv) + \sigma_t(\bx) u(\bx, \bv) = \sigma_s(\bx)\sum_{(n,m )\in \cJ'} \chi_n Y_{nm}(\bv) u_{nm}(\bx) + \sum_{(n,m )\in \cJ'} Y_{nm}(\bv) q_{nm}(\bx)\,.
\end{equation}
Using the same technique as in the two-dimensional case, we can show that the spherical harmonic moments of $u$ satisfy the following integral equation system, $\forall (k,l) \in \cJ'$,
\begin{equation}
u_{kl} (\bx) = \frac{1}{4\pi} \int_{\Omega} \frac{E(\bx, \by)}{|\bx - \by|^2} \left( \sum_{(n,m )\in \cJ'} Y_{nm}\left( \frac{\bx - \by}{|\bx - \by|}\right) Y_{kl}^{\ast}\left( \frac{\bx - \by}{|\bx - \by|}\right) \left(\chi_n  \sigma_s(\by) u_{nm}(\by) + q_{nm}(\by)\right) \right) d\by\,.
\end{equation}
We can introduce a similar vector space $V$, and formulate this integral equation system about $\cU = (u_{kl})_{(k,l)\in\cJ'}$ again into the form
\begin{equation}
	(\cI - \cL\cD) \cU = \cL \cQ
\end{equation}
with the operator matrix $\cL$ having entries
\begin{equation}
\cL_{nm, kl} f (\bx):= \frac{1}{4\pi}\int_{\Omega} \frac{E(\bx, \by)}{|\bx - \by|^2}Y_{nm}\left( \frac{\bx - \by}{|\bx - \by|}\right) Y_{kl}^{\ast}\left( \frac{\bx - \by}{|\bx - \by|}\right) f(\by) d\by,
\end{equation}
the multiplicative operator matrix $\cD$ having only diagonal entries $\cD_{nm, nm} = \chi_n \sigma_s$, and the source vector $\cQ = (q_{nm})_{(n,m)\in\cJ'}$.

\subsection{Elementary properties}

In general, it is \emph{not} guaranteed that the truncated approximations~\eqref{EQ:TRUNC} of the scattering phase function are always non-negative~\cite{Wiscombe-JAS77}. If indeed we have $p_M(\cos\theta) \ge 0$, then we can easily adapt results from standard transport theory, for instance those in~\cite{Agoshkov-Book12}, to show the uniqueness of the solution to the truncated radiative transfer equation~\eqref{eq:RTE3}, which automatically implies the uniqueness of the solution to the integral equation system~\eqref{eq:sys}. We summarize the result in the following theorem whose proof we omit here.
\begin{theorem}
	If the truncated phase function $p_M(\cos\theta)\ge 0$ for all $\theta\in [0,2\pi]$, then the system of integral equations~\eqref{eq:sys} is uniquely solvable in $V$. In this case, the operator $\cL \cD: V\to V$ is a contraction, that is
	\begin{equation}
	\|\cL \cD\|_{op} \le \sup_{\bx\in \Omega} \frac{{\sigma}_{s}(\bx)}{{\sigma}_t(\bx)} \le k_0 < 1\,.
	\end{equation}
	where $\|\cdot\|_{op}$ denotes the operator norm on $V$.
\end{theorem}

The above theorem says that when the truncated scattering phase function $p_M$ is non-negative, the linear operator $(\mathcal{I}-\mathcal{L}\mathcal{D})$ is invertible in $V$, and the solution $\mathcal{U}\in V$ to the integral equation can be found as
\begin{equation}
\label{eq:inte}
\mathcal{U} = (\mathcal{I}-\mathcal{L}\mathcal{D})^{-1} \mathcal{L}\mathcal{Q}\,,
\end{equation}
through either direct or iterative methods. In the two-dimensional case, although the integral operator matrix $\cL$ has $\cO(M^2)$ entries in total, there are only $\cO(M)$ distinct integral operators $\cK_{n}$ due to the Toeplitz structure of the matrix. In the three-dimensional case, we will have to deal with $\cO(M^4)$ different integral operators $\cL_{nm,kl}$ for $(n,m), (k,l)\in\cJ'$ if no further simplification can be made. Direct application of the operator $\cL$ would be very expensive if we could not find good ways to compress the operators. Let us look at this in more detail. Let us first use the contraction rule of spherical harmonics to have the following equality using the $3j$-symbols~\cite{Wigner-Book59}:
\begin{equation}\label{EQ:3J}
Y_{nm}(\bv) Y_{kl}^{\ast}(\bv) = \sqrt{\frac{(2n+1)(2k+1)}{4\pi}} \sum_{r,s}(-1)^{s+l} \sqrt{2r+1} 
\begin{pmatrix}
n & k & r \\ m & -l & -s
\end{pmatrix} \begin{pmatrix}
n & k & r\\0 & 0 & 0
\end{pmatrix} Y_{rs}(\bv)\,,
\end{equation}
where $r$ and $s$ are integers satisfying the $3j$-symbol selection rules~\cite{Wigner-Book59}:
\begin{equation}
|s|\le r, \quad m - l -s = 0, \quad |n-k|\le r \le n + k.
\end{equation} 
Using this equality, i.e.~\eqref{EQ:3J}, we can express the operator $\cL_{nm, kl}$ as:
\begin{equation}\label{EQ:K3}
	\cL_{nm,kl} f(\bx) = \frac{1}{4\pi}\sum_{r,s} \mu_{mls}^{nkr}\int_{\Omega} G_{rs}(\bx, \by)  f(\by)d\by, \quad G_{rs}(\bx, \by) = \frac{E(\bx, \by)}{|\bx - \by|^2} Y_{rs}\left( \frac{\bx - \by}{|\bx - \by|}\right)
\end{equation}
with the constant
\begin{equation}
\mu_{mls}^{nkr} := \sqrt{\frac{(2n+1)(2k+1)(2r+1)}{4\pi}} (-1)^{s+l} 
\begin{pmatrix}
n & k & r \\ m & -l & -s
\end{pmatrix} \begin{pmatrix}
n & k & r\\0 & 0 & 0
\end{pmatrix}\,.
\end{equation}
It is clear that applying each $\cL_{nm,kl}$ on average will involve $\cO(M)$ integrals. Therefore the total complexity to apply $\cL$ is $\cO(M^5)$, while there are only $\cO(M^2)$ distinct integral kernels $G_{rs}$ in total. Numerically, compressing integral operators usually takes much more time than applying the operators. It is therefore more favorable to have less compression when $M$ is kept small.

The number of integral equations and integral kernels in the system~\eqref{eq:inte} depends on the number of terms $M$ in the truncation~\eqref{EQ:TRUNC}. $M$ depends on the accuracy requirement and the smoothness of the scattering kernel in $\theta$. When the scattering kernel is highly anisotropic, large $M$ is needed to get a good approximation since $\chi_n$ decays slowly. In such a case, the integral kernels $G_n(\bx,\by)$ in~\eqref{EQ:K2} and $G_{rs}(\bx,\by)$ in ~\eqref{EQ:K3} become more oscillatory as $n$ and $r$ increases, resulting in larger computational cost in the evaluation of integral operators with such kernels. 

Evaluation of the application of an integral operator in the discretized case is equivalent to the evaluation of a matrix-vector multiplication with a dense matrix. The technique to reduce the computational complexity of such matrix-vector multiplication is to construct (hierarchical) low-rank approximations to the dense matrix involved. When such low-rank approximations are available, it is often the case that one can reduce the complexity of a matrix-vector multiplication to something that is comparable to that of a vector-vector multiplication. Examples of algorithms based on such low-rank approximations are fast multipole methods~\cite{GrRo-JCP87,Greengard-Thesis88} and butterfly algorithms~\cite{MiBo-IEEE96,CaDeYi-MMS09,LuQiBu-JCP14}.
 
In the rest of the paper, we will analyze the separability properties of the kernel functions, for instance $G_n(\bx,\by)$ and $G_{rs}(\bx,\by)$, in the integral equation formulation of the truncated anisotropic radiative transfer equation. Our analysis will provide a mathematical understanding on low-rank approximations of the matrices corresponding to the discretization of these continuous integral kernels.

%
%
\section{Separability of the kernel functions}
\label{SEC:SEP}

We are interested in the approximate separability property of the integral kernels defined in~\eqref{EQ:K2} and ~\eqref{EQ:K3}. For a given function of two variables, $G(\bx, \by)$, we characterize its approximate separability as follows. Take two disjoint sets $X, Y\subset \Omega \subset \bbR^d$ and a tolerance $\eps > 0$, we would like to characterize the smallest number $N^{\eps}$ for which there exists $\Big(f_l(\bx), g_l(\by)\Big), l =1,2,\dots, N^{\eps}$ such that
\begin{equation}\label{EQ:SEP}
\left\|G(\bx, \by) - \sum_{l=1}^{N^{\eps}} f_l(\bx) g_l(\by)\right\|_{L^2(X\times Y)} \le \eps \|G\|_{L^2(X\times Y)}.
\end{equation}
We choose $L^2$ norm in the function space because the approximate separability definition is directly related to the best rank-$r$ approximation of a matrix, a discretized version of $G(\bx,\by)$, which can be computed by the singular value decomposition (SVD). The reason we require the two sets $X$ and $Y$ be disjoint is because the integral kernels have singularities at $\bx=\by$ which make the kernel not square integrable in its domain of definition. This means that the full matrix corresponding to an integral kernel in the computation domain, i.e., $G(\bx,\by), (\bx,\by)\in \Omega \times \Omega$, does not have a low-rank approximation. However, if the kernel is highly separable, i.e., $N^{\eps}$ grows at most (poly-) logarithmically in $\eps$ as $\eps\rightarrow 0$,  for well separated $X$ and $Y$, with proper ordering, i.e., grouping indices into well separated admissible sets, the matrix allows a hierarchical structure for which the off-diagonal sub-matrices have low-rank approximations. For example, it was shown that the Green's functions corresponding to coercive elliptic differential operators in divergence form are highly separable~\cite{BeHa-NM03}, i.e., $N^{\eps}=O(|\log \eps|^{d+1})$, when $X, Y$ are well separated. This property implies that the inverse matrix of the linear system $Ax=b$ resulting from a discretization of the differential equation has a hierarchical low-rank structure. This property has been exploited in developing fast director solvers for coercive elliptic partial differential equations~\cite{Hackbusch-Computing99,Bebendorf-MC05, Borm-NM10, HoYi-CPAM16} using the fact that each column of $A^{-1}$ is a discrete version of the underlying Green's function. On the other hand, it was shown that the Green's function for high frequency Helmholtz equation is not highly separable due to the highly oscillatory phase in the function~\cite{EnZh-CPAM18} and hence hierarchical low-rank approximations do not exist for the inverse matrix of the discretized linear system when the wavenumber is large. The concept has been generalized to study the separability of the covariance function for random fields~\cite{BrZhZh-MMS19} which relates the number of terms needed for the Karhumen--Lo\'{e}ve expansion of a random field to a given accuracy requirement.

\begin{remark}
If we view the function $G(\bx,\by)$ as a family of functions of $\bx$ in $L^2(X)$ parametrized by $\by\in Y$ (or vice versa), then the concept of approximate separability can be directly related to the concept of Kolmogorov $n$-width\footnote{The Kolmogorov $n$-width of a set $S$ in a normed space $W$ is its worst-case distance to the best $n$-dimensional linear subspace $L_n$:
\[
	d_n(S, W) :=\inf_{L_n}\sup_{f\in S}\inf_{g\in L_n} \|f-g\|_W\,.
\]
} 
for this family of functions~\cite{Kolmogorov-AM36}. Any linear subspace in the function space that approximates this family of functions within relative tolerance $\eps$ has a dimension of at least $N^{\eps}$, and the space spanned by $f_l(\bx), l=1, 2, \ldots, N^{\eps}$ in~\eqref{EQ:SEP} is an optimal one. The Kolmogorov $n$-width of a set reveals its intrinsic complexity. 
\end{remark}

In the next two subsections, we show the approximate separability property of the integral kernels~\eqref{EQ:K2} (in 2D) ~\eqref{EQ:K3} (in 3D) in the following two scenarios:
\begin{itemize}
\item For a fixed $\eps$, we derive the lower bound for $N^{\eps}$ to show how it grows as $n$ (resp. $r$) increases in 2D (resp. 3D) due to the increasing oscillations in the kernel.
\item For a fixed $n$ (resp. $r$ in 3D), we derive the upper bound for $N^{\eps}$ to show how it grows as $\eps \rightarrow 0$.
\end{itemize}
The main procedure for the derivation is the same in 2D and 3D. However, as we will see, the calculations in 3D are much more complicated.

\subsection{Separability's lower bounds}
\label{SEC:LOWER}

\subsubsection{The two-dimensional case}

Let us first estimate the separability for the integral kernel $G_n$ defined in~\eqref{EQ:K2} for large $n$ in 2D.  Analogous to the phenomenon shown in~\cite{EnZh-CPAM18}, for a fixed tolerance $\eps$, the number of terms in the separable approximation of $G_n$ has to grow as some power of $n$ due the fast oscillation of the kernel. For simplicity, we assume that the coefficients ${\sigma}_{t}$ and ${\sigma}_s$ are smooth over the physical domain $\Omega$. 

Following the idea in~\cite{EnZh-CPAM18}, we first characterize the correlation between two integral kernels $G_n(\cdot, \by_1)$ and $G_n(\cdot, \by_2)$ with $\by_1, \by_2\in Y$
\begin{equation}
\cC (\by_1, \by_2) = \frac{1}{\|G_n(\cdot, \by_1)\|_2 \|G_n(\cdot, \by_2)\|_2} \int_X G_n(\bx, \by_1) \overline{G_n(\bx, \by_2)} d\bx\,.
\end{equation}
\begin{lemma}\label{THM:LOW} 
	Let $X, Y$ be two disjoint compact domains in $\bbR^2$ and $\text{\rm dist}(X,Y) \ge \gamma\, \text{\rm diam}(Y)$ for some $\gamma = O(1)$. Then there exists $\frac{3}{2}\ge  \alpha \ge 1$ such that
	\begin{equation}
	\left|\cC(\by_1, \by_2)\right| \le O\left( (n |\by_1 -\by_2|)^{-\alpha} \right),  \quad \mbox{as }\ \ n |\by_1 -\by_2|\to\infty\,.
	\end{equation}
\end{lemma}
\begin{proof}
	First, $\|G_n(\cdot, \by) \|_2$ is a smooth function of $\by$ since the fast oscillation phase is not present and there exist positive constants $C_1$ and $C_2$, independent of $n$, such that $C_1 < \|G_n(\cdot, \by)\|_2 < C_2$. For the integral part, let us introduce
	\begin{equation}
	\begin{aligned}
	\tilde{n} &= n|\by_1 - \by_2|\,,\\
	\phi(\bx) &= \frac{\arg(\bx - \by_2) - \arg(\bx - \by_1)}{|\by_1 - \by_2|}\,,\\
	u(\bx) &= \frac{E(\bx, \by_1) E(\bx, \by_2)}{\|G_n(\cdot, \by_1)\|_2 \|G_n(\cdot, \by_2)\|_2|\bx-\by_1| |\bx - \by_2|}\,.
	\end{aligned}
	\end{equation} 
	The sign of $\phi$ depends on the locations of $\bx$. The correlation can then be written as the following oscillatory integral,
	\begin{equation}
	\left|\cC(\by_1, \by_2)\right| = \left| \int_X e^{i\tilde{n} \phi(\bx)} u(\bx) d\bx  \right|\,.
	\end{equation}
	We also have $|\phi(\bx)|\le O(1/\text{\rm dist}(X,Y))$ and $|\nabla \phi(\bx)|\neq 0$ unless $\bx$ sits inside the segment between $\by_1, \by_2$. Therefore, there is no stationary point in our setting and $|\nabla \phi(\bx)| \ge O(\gamma^{-2}), \forall x\in X$. $u(\bx)$ is smooth since $E(\bx, \by)$ is smooth. We define the differential operator $L$:
	\begin{equation}
	L = \frac{1}{|\nabla \phi|^2} \nabla \phi \cdot \nabla,\quad L^{\ast} = -\nabla \cdot \frac{1}{|\nabla \phi|^2} \nabla \phi\,.
	\end{equation}
	Using integration by part, we  have
	\begin{equation}
	\begin{aligned}
	&\int_X e^{i\tilde{n} \phi(\bx)} u(\bx) d\bx = \frac{1}{i \tilde{n}} \int_X (L e^{i \tilde{n}\phi(\bx)}) u(\bx) d\bx \\
	&= \frac{1}{i\tilde{n}}\left[ \int_{X} e^{i\tilde{n}\phi(\bx)} (L^{\ast} u(\bx)) d\bx + \int_{\partial X} |\nabla \phi(\bx)|^{-2} \bn(\bx)\cdot \nabla \phi(\bx) e^{i\tilde{n}\phi(\bx)} u(\bx) dS(\bx) \right] \\
	&=-\frac{1}{\tilde{n}^2}\left[ \int_X e^{i\tilde{n}\phi(\bx)} ((L^{\ast})^2 u(\bx) )d\bx + \int_{\partial X} |\nabla \phi(\bx)|^{-2} \bn(\bx)\cdot \nabla\phi(\bx)  e^{i \tilde{n}\phi(\bx)}L^{\ast} u(\bx) dS(\bx)\right] \\
	&\quad + \frac{1}{i\tilde{n}}\int_{\partial X} |\nabla \phi(\bx)|^{-2} \bn(\bx)\cdot \nabla \phi(\bx) e^{i\tilde{n}\phi(\bx)} u(\bx) dS(\bx)\,.
	\end{aligned} 
	\end{equation}
	The last term has leading order and is an oscillatory integral along the boundary $\partial X$. If $\phi(\bx)$ has isolated non-degenerate critical points on $\partial X$, which is a one dimensional curve, then the boundary integral is of order $O(\tilde{n}^{-\frac{1}{2}})$ by stationary phase theory. So we have 
	\begin{equation}
	\left| \int_X e^{i\tilde{n} \phi(\bx)} u(\bx) d\bx\right| \le  O(\tilde{n}^{-3/2}),\quad  \tilde{n}\to\infty\,.
	\end{equation}
	In the special case that $\phi(\bx)$ has non-isolated critical points, which means that $\partial X$ and some level set of $\phi$ have coincidental part, the boundary integral will be at order of $O(1)$ and 
	\begin{equation}
	\left| \int_X e^{i\tilde{n} \phi(\bx)} u(\bx) d\bx\right| \le  O(\tilde{n}^{-1}),\quad \tilde{n}\to \infty\,.
	\end{equation}
	This completes the proof.
\end{proof}
Figure~\ref{fig:1} shows numerical evidence of the scaling of the (normalized) correlation function $\cC(\by_1, \by_2)$ with respect to $\tilde n$ at two different locations. In both cases, we observed the expected decay behavior as predicted by the theory, i.e $O(\tilde{n}^{-3/2})$ for general case and $O(\tilde{n}^{-1})$ if part of $\partial X$ coincides with the level set of $\phi$.
\begin{figure}[!t]
	\centering
	\begin{tabular}{cc}
		\includegraphics[scale=0.38]{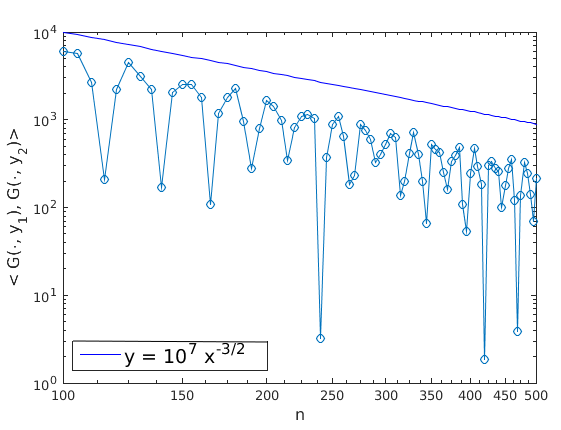}
		&
		\includegraphics[scale=0.38]{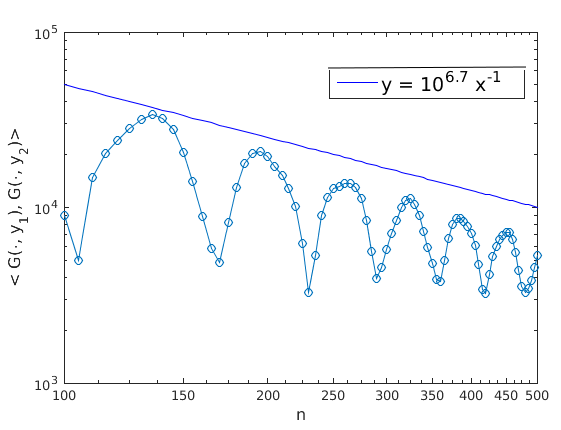}
	\end{tabular}
	\caption{The normalized correlation of two kernel functions at two different locations $\by_1$ and $\by_2$. $X$ is unit square $[0,1]\times[0,1]$. Left: $\by_1 = (2,0), \by_2=(2, \frac{1}{2})$. In this case, $\phi$ has isolated non-degenerate critical points on $\partial X$. Right: $\by_1 = (2,0), \by_2 = (\frac{3}{2}, 0)$.  In this case, all points on the bottom side of $\partial X$ are critical points of $\phi$. }
	\label{fig:1}
\end{figure}

Using the above correlation estimate with $\alpha \ge \frac{d}{2}$ where $d=2$, by Lemma 3.1 and Theorem 3.1 of~\cite{EnZh-CPAM18}, we directly conclude the following lower bound for the separability of the integral kernel $G_n$.
\begin{theorem}\label{thm:sep}
	Let $X, Y$ be two disjoint compact domains in $\bbR^2$ and $\text{\rm dist}(X,Y) \ge \gamma\, \text{\rm diam}(Y)$ for some $\gamma = O(1)$, then for any $\eps > 0$, if there are functions $f_l(\bx)\in L^p(X)$,  $g_l(\by) \in L^p(Y)$, $p=2$ or $p=\infty$, $l=1,\dots, N^{\eps}$ such that
	\begin{equation}\label{EQ:SEP2}
	\left\|G_n(\bx, \by) - \sum_{l=1}^{N^{\eps}} f_l(\bx) g_l(\by)\right\|_{L^p(X\times Y)} \le \eps \|G_n\|_{L^p(X\times Y)},\quad \bx\in X, \by\in Y\,,
	\end{equation}
	then $N^{\eps} \ge O (n^{2-\delta})$ for any small $\delta > 0$ when $n$ is sufficiently large. 
\end{theorem}

\begin{figure}[!htb]
    \centering
    \begin{tabular}{cc}
        \includegraphics[scale=0.6]{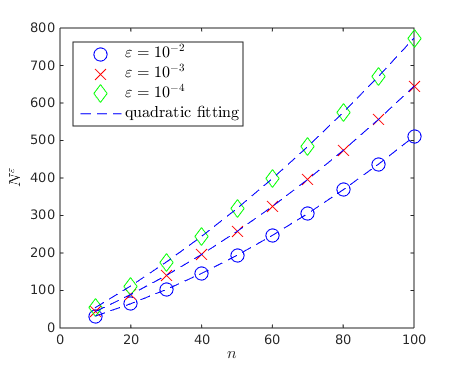}&
        \includegraphics[scale=0.6]{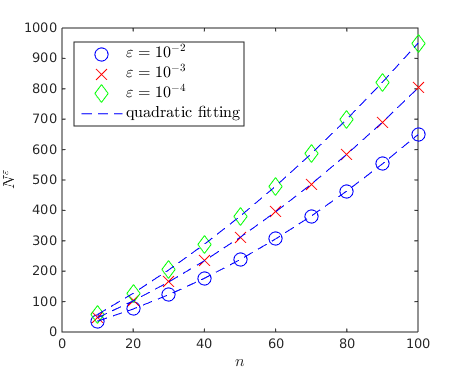}
    \end{tabular}
    \caption{The growth of the number of  leading singular values of the matrix $\big(G_n(\bx_i, \by_j)\big)$ above different threshold $\eps$ with respect to $n$ in the two-dimensional case. Left: $\bx_i\in [0\ 1]^2$ and $\by_j\in [1.25\ 2.25]\times[0.5\ 1.5]$; Right: $\bx_i\in [0\ 1]^2$ and $\by_j\in [1.25\ 2.25]\times[0\ 1]$.}
    \label{fig:2}
\end{figure}
The growth of $N^\eps$ with respect to $n$ in the above theorem is a direct manifest of the growth of the number of leading singular values of the corresponding integral operator above a certain threshold. In Figure~\ref{fig:2}, we show the number of the leading singular values of the matrix $(a_{ij}):=\big(G_n(\bx_i, \by_j)\big)$ above different threshold $\eps$ with respect to $n$. In the plots, we take $\bx_i$ and $\by_j$ from uniformly distributed grid points in $X$ and $Y$ respectively, with the grid size resolves the length scale of $n^{-1}$.  The left plot shows the result for two unit squares $X$ and $Y$ at centers of $(0.5, 0.5)$ and $(1.75, 1.0)$ respectively while the the right plot shows the result for two unit squares at centers $(0.5, 0.5)$ and $(1.75, 0.5)$ respectively. In both plots, we observe quadratic growth of the number of leading singular values above a certain threshold with respect to $n$.


\subsubsection{The three-dimensional case}

\begin{figure}[!htb]
    \centering 
    \includegraphics[scale=0.17]{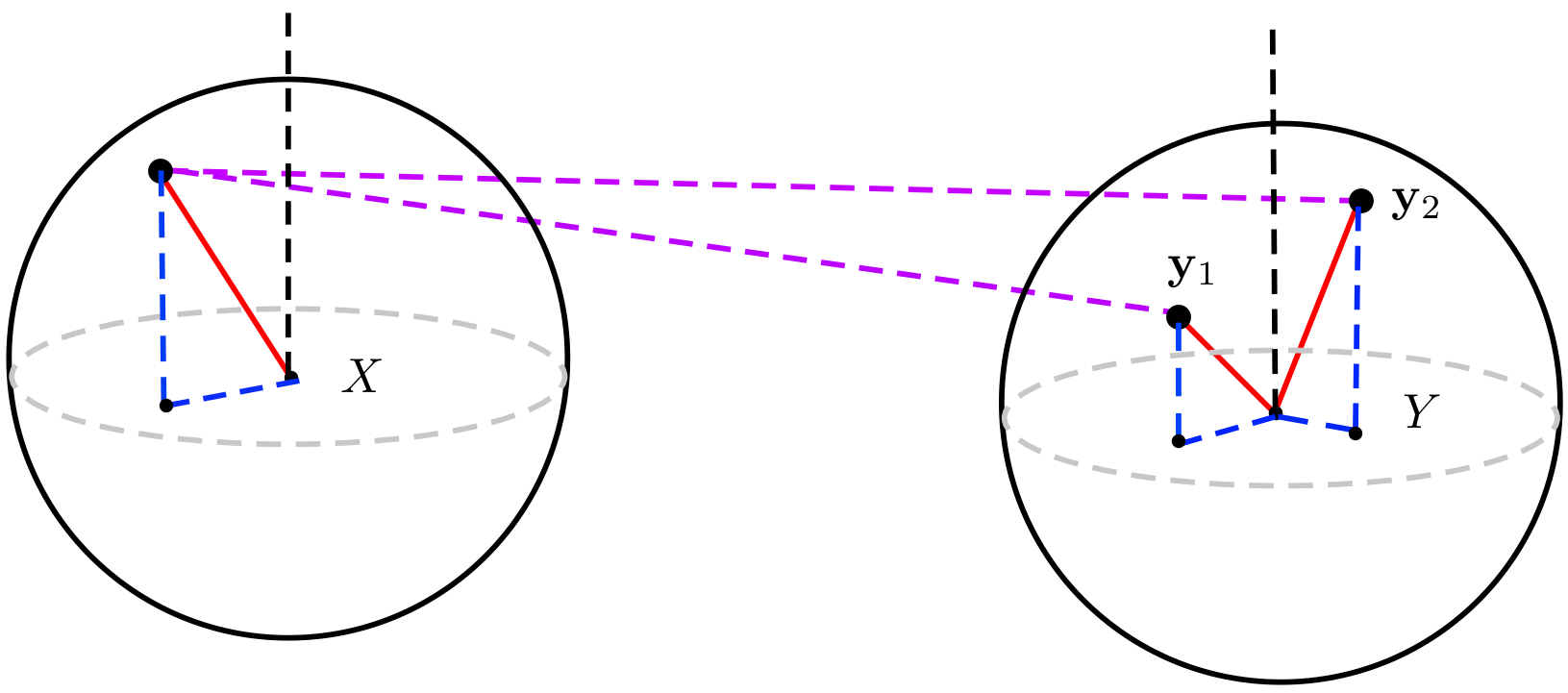}
    \caption{The sets $X$ and $Y$ are compact manifolds. The correlation function estimates the inner product in~\eqref{EQ:COR} as $n\to\infty$.}
    \label{FIG:COR}
\end{figure}
Let $X$ and $Y$ be two disjoint convex compact domains in $\bbR^3$. In the same manner as in the 2D case, to study approximate separability of the kernel $G_{nm}$ in~\eqref{EQ:K3}, we first characterize the correlation of two integral kernels $G_{nm}(\cdot, \by_1)$ and $G_{nm}(\cdot,\by_2)$ for $\by_1, \by_2\in Y$; see Figure~\ref{FIG:COR} for an illustration of the setup:
\begin{equation}\label{EQ:COR}
\cC (\by_1, \by_2) = \frac{1}{\|G_{nm}(\cdot, \by_1)\|_2 \|G_{nm}(\cdot, \by_2)\|_2} \int_X G_{nm}(\bx, \by_1) \overline{G_{nm}(\bx, \by_2)} d\bx\,.
\end{equation}
The separability of $G_{nm}$ depends on the behavior of the spherical harmonics $Y_{nm}$. Here we present the case when $m=0$.
In this case, the spherical harmonics actually is closely related to the Legendre polynomial of degree $n$:
\begin{equation}\label{EQ:YN0}
Y_{n0}(\theta, \phi) = \sqrt{\frac{2n+1}{4\pi}} P_n(\cos\theta)\,.
\end{equation}
We first show some basic properties of $Y_{n0}$.
\begin{lemma}\label{LEM:YN0}
	Let $f\in L^{\infty}(\bbS^2)$, then 
	\begin{equation}
	\begin{aligned}
	\lim_{n\to\infty} \int_0^{2\pi} \int_0^{\pi} |Y_{n0}(\theta, \phi)|^2 f(\theta, \phi) \sin\theta d\theta d\phi = \frac{1}{2\pi^2} \int_0^{\pi} \int_{0}^{2\pi} f(\theta, \phi) d\phi d\theta\,.
	\end{aligned}
	\end{equation}
\end{lemma}
\begin{proof}
From~\eqref{EQ:YN0}, we find
\begin{equation}\label{EQ:LEG}
\int_0^{2\pi} \int_0^{\pi} |Y_{n0}(\theta, \phi)|^2 f(\theta, \phi) \sin\theta d\theta d\phi  = \frac{2n+1}{2}\int_0^{\pi}  |P_n(\cos\theta)|^2 \tilde{f}(\theta)\sin\theta d\theta \,,
\end{equation}
where $\tilde{f}(\theta) = \frac{1}{2\pi}\int_0^{2\pi} f(\theta, \phi) d\phi$. Then by the classical Darboux formula, see for instance Theorem 8.21.13 of~\cite{Szeg-Book39}, for a fixed absolute constant $c > 0$, when $n\to\infty$ and $\theta \in (\frac{c}{n}, \pi - \frac{c}{n})$, we have
\begin{equation}
P_n(\cos\theta)  = \sqrt{\frac{2}{n\pi\sin\theta}} \left[ \cos \left((n + \frac{1}{2})\theta - \frac{\pi}{4} \right) + \frac{\cO(1)}{n \sin\theta}\right]\,.
\end{equation}
Therefore, for  $\theta \in (\frac{c}{n}, \pi - \frac{c}{n})$,  we have
\begin{equation}
|P_n(\cos\theta)|^2 \sin\theta = \frac{2}{n\pi} \cos^2\left((n + \frac{1}{2})\theta - \frac{\pi}{4} \right) + \frac{\cO(1)}{n^2\sin\theta} + \frac{\cO(1)}{n^3\sin^2\theta}\,.
\end{equation}
We decompose the integral~\eqref{EQ:LEG} into three parts:
\begin{equation}
\frac{2n+1}{2}\int_0^{\pi}  |P_n(\cos\theta)|^2 \tilde{f}(\theta)\sin\theta d\theta = \frac{2n+1}{2} \left[ \int_{c/n}^{\pi-c/n} + \int_{\pi-c/n}^{\pi}+\int_0^{c/n} \right]  |P_n(\cos\theta)|^2 \tilde{f}(\theta)\sin\theta d\theta\,.
\end{equation}
Using the fact $|P_n(\cos\theta)|\le 1$, we can estimate the following
\begin{equation}
\begin{aligned}
\Big|\int_{0}^{c/n} |P_n(\cos\theta)|^2 \tilde{f}(\theta) \sin\theta d\theta \Big| &\le \int_{0}^{c/n} |\tilde{f}(\theta)| \sin\theta d\theta \le  \frac{c^2}{n^2}\|f\|_{L^{\infty}(\bbS^2)} = \cO\left(\frac{1}{n^2}\right).
\end{aligned}
\end{equation}
The same estimate applies to the integral over $[\pi - c/n, \pi]$. On the other hand, since $\theta \le \frac{\pi}{2}\sin\theta$ for $\theta \in [0,\frac{\pi}{2}]$, we find that
\begin{equation}
\begin{aligned}
\Big|\int_{\frac{c}{n}}^{\pi-\frac{c}{n}} \frac{1}{n\sin\theta} \tilde{f}(\theta) d\theta  \Big| &\le 2\|f\|_{L^{\infty}(\bbS^2)}\int_{\frac{c}{n}}^{\pi/2} \frac{\pi}{2 n\theta} d\theta =  \cO\left(\frac{\log n}{n}\right)\,,\\
\Big|\int_{\frac{c}{n}}^{\pi-\frac{c}{n}} \frac{1}{n^2\sin^2\theta} \tilde{f}(\theta) d\theta  \Big| &\le 2\|f\|_{L^{\infty}(\bbS^2)}\int_{\frac{c}{n}}^{\pi/2} \frac{\pi^2}{4 n^2\theta^2} d\theta =  \cO\left(\frac{1}{n}\right)\,.
\end{aligned}
\end{equation}
Therefore
\begin{equation}
\begin{aligned}
\frac{2n+1}{2}\int_0^{\pi}  |P_n(\cos\theta)|^2 \tilde{f}(\theta)\sin\theta d\theta &= \frac{2n+1}{n\pi} \int_{\frac{c}{n}}^{\pi-\frac{c}{n}}  |P_n(\cos\theta)|^2 \tilde{f}(\theta)\sin\theta d\theta + \cO\left(\frac{1}{n}\right) \\
&=\frac{2n+1}{n\pi}\int_{\frac{c}{n}}^{\pi-\frac{c}{n}} \left[\frac{1}{2} + \frac{1}{2}  \sin \left((2n + 1)\theta \right)\right]\tilde{f}(\theta) d\theta + \cO\left(\frac{1}{n}\right)\\
&\quad + \int_{\frac{c}{n}}^{\pi-\frac{c}{n}} \frac{\cO(1)}{n\sin\theta} \tilde{f}(\theta) d\theta  + \int_{\frac{c}{n}}^{\pi-\frac{c}{n}} \frac{\cO(1)}{n^2\sin^2\theta} \tilde{f}(\theta) d\theta  \\
&= \frac{2n+1}{n\pi}\int_{\frac{c}{n}}^{\pi-\frac{c}{n}} \left[\frac{1}{2} + \frac{1}{2}  \sin \left((2n + 1)\theta \right)\right]\tilde{f}(\theta) d\theta + \cO\left(\frac{\log n}{n}\right)\\
&= \frac{2n+1}{n\pi}\int_{0}^{\pi} \left[\frac{1}{2} + \frac{1}{2}  \sin \left((2n + 1)\theta \right)\right]\tilde{f}(\theta) d\theta + \cO\left(\frac{\log n}{n}\right)\,.
\end{aligned} 
\end{equation}
Using the Riemann-Lebesgue lemma on the integral of $\tilde{f}(\theta)\sin((2n+1)\theta)$, we obtain
\begin{equation}
\lim_{n\to\infty}  \frac{2n+1}{2}\int_0^{\pi}  |P_n(\cos\theta)|^2 \tilde{f}(\theta)\sin\theta d\theta =\lim_{n\to\infty} 
\frac{2n+1}{2n\pi} \int_0^{\pi} \tilde{f}(\theta) d\theta = \frac{1}{\pi} \int_0^{\pi} \tilde{f}(\theta) d\theta\,.
\end{equation}
The proof is complete.
\end{proof}
\begin{lemma}\label{LEM:LOW}
	Let $X, Y$ be two disjoint compact domains in $\bbR^3$,  we have
	\begin{equation}
	\|G_{n0}(\cdot, \by)\|_2 = \cO(1), \quad n\to\infty\,.
	\end{equation}
\end{lemma}
\begin{proof}
	By definition, we have
	\begin{equation}
	\begin{aligned}
	\int_{X} G_{n0}(\bx, \by)\overline{G_{n0}(\bx, \by)} d\bx =  \int_{X} \left(\frac{E(\bx, \by)}{|\bx - \by|^2}\right)^2 Y_{n0}\left(\frac{\bx - \by}{|\bx - \by|}\right) Y_{n0}^{\ast}\left(\frac{\bx - \by}{|\bx - \by|}\right)  d\bx\,.
	\end{aligned}
	\end{equation}  
	We now do a change of coordinate by shifting $\by$ to the origin and transforming $\bx$ back to spherical coordinate $\bx = (r, \theta, \phi)$. We denote by $V$ the transformed domain of $X$. Using the fact that $E(\bx, \by)$ is bounded from below (since $X$ and $Y$ are compact), we only need to estimate the following
	\begin{equation}
	\begin{aligned}
	\int_{X} \frac{1}{|\bx - \by|^4} Y_{n0}\left(\frac{\bx - \by}{|\bx - \by|}\right) Y_{n0}^{\ast}\left(\frac{\bx - \by}{|\bx - \by|}\right)  d\bx &=  \int_{V} \frac{1}{r^4} |Y_{n0}(\theta, \phi)|^2 r^2 dr \sin \theta d\theta d\phi  \\
	&= \int_{V} \frac{1}{r^2} |Y_{n0}(\theta, \phi)|^2  dr \sin \theta d\theta d\phi \\
	& = \int_0^{2\pi} \int_0^{\pi} \left(\int_{0}^{\infty} \frac{\chi_{V}}{r^2}  dr\right) |  Y_{n0}(\theta, \phi)|^2  \sin \theta d\theta d\phi \,,
	\end{aligned}
	\end{equation}
	where $\chi_V$ is the characteristic function of $V$. The $\cO(1)$ upper bound can be immediately concluded since $X$ and $Y$ are disjoint, which means $\chi_V \equiv 0$ when $r$ is close to zero. When $n\to\infty$,  use the Lemma~\ref{LEM:YN0}, 
	\begin{equation}
	\begin{aligned}
	\lim_{n\to\infty}\int_0^{2\pi} \int_0^{\pi} \left(\int_{0}^{\infty} \frac{\chi_{V}}{r^2}  dr\right) |  Y_{n0}(\theta, \phi)|^2  \sin \theta d\theta d\phi  &= \frac{1}{2\pi^2}\int_{0}^{2\pi} \int_0^{\pi}	\left(\int_{0}^{\infty} \frac{\chi_{V}}{r^2}  dr\right) d\theta d\phi\,.
	\end{aligned}
	\end{equation}
	On the other hand, we know that
	\begin{equation}
	\frac{1}{2\pi^2}\int_{0}^{2\pi} \int_0^{\pi}	\left(\int_{0}^{\infty} \frac{\chi_{V}}{r^2}  dr\right) d\theta d\phi\ge  \frac{1}{2\pi^2}\int_{0}^{2\pi} \int_0^{\pi}	\int_{0}^{\infty} \frac{\chi_{V}}{r^4}  r^2 dr\sin\theta d\theta d\phi \\
	\ge \frac{|X|}{2\pi^2 {\ell}^4}\,,
	\end{equation}
	where ${\ell } = \sup_{\bx\in X, \by\in Y}|\bx - \by|$ and $|X|$ is the volume of $X$.
	Therefore $\|G_{n0}(\cdot, \by)\|_2^2$ is $\cO(1)$, $\forall \by\in Y$ when $n$ is sufficiently large. 
\end{proof}	

\begin{lemma}\label{LEM:3D1}
	Let $X$ and $Y$ be two disjoint convex compact domains in $\mathbb{R}^3$. Suppose $X$ and $Y$ have disjoint projections onto the $xy$-plane and $z$-axis,
    then for the correlation function 
	\begin{equation}
	\cC(\by_1, \by_2)=\frac{1}{\|G_{n0}(\cdot, \by_1)\|_2 \|G_{n0}(\cdot, \by_2)\|_2} \int_X G_{n0}(\bx, \by_1) \overline{G_{n0}(\bx, \by_2)} d\bx\,,
	\end{equation}
	there exists constant $1\le \alpha\le 2$ such that when $n|\by_1 - \by_2|\to \infty$, $\cC(\by_1, \by_2)$ satisfies 
	\begin{equation}\label{EQ: ALPHA}
	|\cC(\by_1, \by_2)| \le  \cO((n|\by_1 - \by_2|)^{-\alpha})\,.
	\end{equation} 
\end{lemma}
\begin{proof}
From Lemma~\ref{LEM:LOW}, the correlation function $\cC(\by_1, \by_2)$'s decay rate only depends on the integral part. We represent $(\bx - \by_i)$ in spherical coordinate as $(r_i, \theta_i, \varphi_i)$, $i=1,2$. From the assumption that $X$ and $Y$ have disjoint projections onto $xy$-plane, 
$\theta_i\in [c', \pi - c']$  for some fixed $c', \frac{\pi}{2} > c' > 0$. From the convexity, $|\varphi_1 - \varphi_2| < \pi - c''$ for some fixed $c'' > 0$.  Using the approximation for Legendre polynomial in Theorem 8.21.4 of~\cite{Szeg-Book39}, we have
	 \begin{equation}\label{EQ:LEGEND}
	 \begin{aligned}
	 P_n(\cos\theta) &= \left(\frac{2}{\pi \hat{n} \sin\theta}\right)^{1/2} \cos\left(\hat{n}\theta - \frac{\pi}{4}\right) \\&+ \frac{\cos\theta}{8}\left(\frac{2}{\pi\hat{n}^3\sin^3\theta}\right)^{1/2}\cos\left(\hat{n}\theta- \frac{3}{4}\pi\right)+\cO\left(\frac{1}{{\hat{n}}^{5/2}}\right)\,,
	 \end{aligned}
	 \end{equation}
	 where $\hat{n} = n + \frac{1}{2}$. Since $\theta_i \in [c', \pi - c']$, $\sin\theta_1$ and $\sin\theta_2$ are both $\cO(1)$. Therefore,
 \begin{equation}
 \begin{aligned}
 \int_X G_{n0}(\bx, \by_1) \overline{G_{n0}(\bx, \by_2)} d\bx &= \frac{\hat{n}}{2\pi}\int_X \frac{E(\bx, \by_1) E(\bx, \by_2)}{|\bx - \by_1|^2 |\bx - \by_2|^2} P_{n}(\cos\theta_1) P_{n}(\cos\theta_2) d\bx  \\
 &= \cL_1 + \cL_2 + \cL_3 + \cO(\hat{n}^{-2})\,,
 \end{aligned}
 \end{equation}
 where $\cL_1, \cL_2$ and $\cL_3$ are 
 \begin{equation}
 \begin{aligned}
 \cL_1 &= \frac{1}{\pi^2} \int_X \frac{E(\bx, \by_1) E(\bx, \by_2)}{|\bx - \by_1|^2 |\bx - \by_2|^2}\frac{1}{\sqrt{\sin\theta_1\sin\theta_2}} \cos\left(\hat{n}\theta_1 - \frac{\pi}{4}\right) \cos\left(\hat{n}\theta_2 - \frac{\pi}{4}\right)  d\bx\,, \\
 \cL_2 &= \frac{1}{8\hat{n}\pi^2} \int_X \frac{E(\bx, \by_1) E(\bx, \by_2)}{|\bx - \by_1|^2 |\bx - \by_2|^2}\frac{\cos\theta_2}{\sqrt{\sin\theta_1\sin^3\theta_2}} \cos\left(\hat{n}\theta_1 - \frac{\pi}{4}\right) \cos\left(\hat{n}\theta_2 - \frac{3\pi}{4}\right) d\bx\,, \\
 \cL_3 &= \frac{1}{8\hat{n}\pi^2} \int_X \frac{E(\bx, \by_1) E(\bx, \by_2)}{|\bx - \by_1|^2 |\bx - \by_2|^2}\frac{\cos\theta_1}{\sqrt{\sin^3\theta_1\sin\theta_2}} \cos\left(\hat{n}\theta_1 - \frac{3\pi}{4}\right) \cos\left(\hat{n}\theta_2 - \frac{\pi}{4}\right) d\bx\,.
 \end{aligned}
 \end{equation}
To simplify the presentation, we again introduce the new variables:
 \begin{equation}\label{EQ:DEF}
 \begin{aligned}
 \tilde{n} &= \hat{n} |\by_1 - \by_2|\,,\\
 \tilde{\phi}(\bx) &= \frac{\theta_1 - \theta_2}{|\by_1 - \by_2|},\\
 \hat{\phi}(\bx) & = \theta_1 + \theta_2\,,\\
 u(\bx) &= \frac{E(\bx, \by_1)E(\bx, \by_2)}{\pi^2|\bx - \by_1|^2 |\bx - \by_2|^2}\frac{1}{\sqrt{\sin\theta_1 \sin\theta_2}}\,,\\
 v(\bx) &= \frac{E(\bx, \by_1)E(\bx, \by_2)}{8\pi^2|\bx - \by_1|^2 |\bx - \by_2|^2}\frac{\cos\theta_2}{\sqrt{\sin\theta_1 \sin^3\theta_2}}\,,\\
 w(\bx) &=\frac{E(\bx, \by_1)E(\bx, \by_2)}{8\pi^2|\bx - \by_1|^2 |\bx - \by_2|^2}\frac{\cos\theta_1}{\sqrt{\sin^3\theta_1 \sin\theta_2}}\,.
 \end{aligned}
 \end{equation}
 It is then clear that $|\tilde{\phi}| \le \cO(1/\operatorname{\rm dist}(X,Y))$ and $\hat{\phi} \in[2c', 2\pi-2c']$.  Obviously $\hat{\phi}$ does not have any stationary points in $X$, which means there exists a positive constant $\hat{c}>0$ such that $|\nabla \hat{\phi}|>\hat{c}$ in $X$. For $\tilde{\phi}$,
we compute $|\nabla \tilde{\phi}|^2$ as follows,
 \begin{equation}
 |\nabla \tilde{\phi}|^2 = \frac{1}{|\by_1 - \by_2|^2}\left( \frac{1}{r_1^2} + \frac{1}{r_2^2} - \frac{2}{r_1r_2}\left[ \cos\theta_1 \cos\theta_2 \cos(\varphi_1 - \varphi_2) + \sin\theta_1\sin\theta_2\right] \right).
 \end{equation}
 Since $|\varphi_1 - \varphi_2| < \pi - c''$, the stationary points happen only when $\theta_1 = \theta_2 = \frac{\pi}{2}$ and $r_1 = r_2$.  Because $X$ and $Y$ have disjoint projections on $z$-axis, then these stationary points do not appear either, which means there exists a positive constant $\tilde{c} > 0$ such that $|\nabla \tilde{\phi}| > \tilde{c}$ in $X$ as well.

Because $\theta_1, \theta_2$ are away from both $0$ and $\pi$, then $u, v, w$ are smooth functions for $\bx \in X$ and $\by\in Y$. The integrals $\cL_1$, $\cL_2$ and $\cL_3$ are represented as
\begin{equation}
\begin{aligned}
\cL_1(\bx) &= \frac{1}{2}\Re \left[\int_{X} u(\bx) e^{i\tilde{n}\tilde{\phi}(\bx)}d\bx \right]+ \frac{1}{2}\Im\left[\int_X u(\bx) e^{i \hat{n} \hat{\phi}(\bx)} d\bx \right]\,, \\
\cL_2(\bx) &= -\frac{1}{2\hat{n}}\Im \left[\int_{X} v(\bx) e^{i\tilde{n}\tilde{\phi}(\bx)}d\bx \right]- \frac{1}{2\hat{n}}\Re\left[\int_X v(\bx) e^{i \hat{n} \hat{\phi}(\bx)}d\bx \right]\,, \\
\cL_3(\bx) &= \frac{1}{2\hat{n}}\Im \left[\int_{X} w(\bx) e^{i\tilde{n}\tilde{\phi}(\bx)}d\bx \right]- \frac{1}{2\hat{n}}\Re\left[\int_X w(\bx) e^{i \hat{n} \hat{\phi}(\bx)}d\bx \right]\,.
\end{aligned}
\end{equation}
Similar to the proof of Lemma~\ref{THM:LOW}, let us introduce the operators:
\begin{equation}\label{EQ: OPERATOR}
\begin{aligned}
\tilde{L} = \frac{1}{|\nabla\tilde{\phi}|^2}[\nabla \tilde{\phi}]\cdot \nabla,\quad \tilde{L}^{\ast} = -\nabla \cdot \frac{1}{|\nabla \tilde{\phi}|^2} \nabla \tilde{\phi}\,,
\end{aligned}
\end{equation}
We can verify that
\begin{equation}\label{EQ:INTGL BY PART}
\begin{aligned}
\int_{X} u(\bx) e^{i\tilde{n}\tilde{\phi}(\bx)}d\bx &= \frac{1}{i\tilde{n}} \int_X \left[\tilde{L} e^{i\tilde{n}\tilde{\phi} (\bx)}\right] u(\bx) d\bx\\
&=-\frac{1}{\tilde{n}^2}\left[ \int_X  e^{i\tilde{n}\tilde{\phi}(\bx)}[\tilde{L}^{\ast}]^2 u(\bx) d\bx  + \int_{\partial X} |\nabla\tilde{\phi}(\bx)|^{-2} \bn(\bx)\cdot \nabla \tilde{\phi}(\bx) e^{i\tilde{n}\tilde{\phi}(\bx)} \tilde{L}^{\ast} u(\bx) dS(\bx)\right]\\
&\quad + \frac{1}{i\tilde{n}}\int_{\partial X} |\nabla\tilde{\phi}(\bx)|^{-2}\bn(\bx)\cdot \nabla \tilde{\phi}(\bx) e^{i\tilde{n} \tilde{\phi}(\bx)} u(\bx) dS(\bx) \,.
\end{aligned}
\end{equation}
For the first term on right-hand-side of~\eqref{EQ:INTGL BY PART}, we have the estimates
\begin{equation}
\frac{1}{\tilde{n}^2}\left[ \int_X  e^{i\tilde{n}\tilde{\phi}(\bx)}[\tilde{L}^{\ast}]^2 u(\bx) d\bx  + \int_{\partial X} |\nabla\tilde{\phi}(\bx)|^{-2} \bn(\bx)\cdot \nabla \tilde{\phi}(\bx) e^{i\tilde{n}\tilde{\phi}(\bx)} \tilde{L}^{\ast} u(\bx) dS(\bx)\right] \le \cO(\tilde{n}^{-2}),
\end{equation}
The second term is an oscillatory integral on the surface $\partial X$. When $\tilde{\phi}$ has only non-degenerated isolated stationary points on $\partial X$, from stationary phase theory, 
\begin{equation}\label{EQ:BD3}
\Big|\frac{1}{i\tilde{n}}\int_{\partial X} |\nabla\tilde{\phi}(\bx)|^{-2}\bn(\bx)\cdot \nabla \tilde{\phi}(\bx) e^{i\tilde{n} \tilde{\phi}(\bx)} u(\bx) dS(\bx)\Big| \le \cO(\tilde{n}^{-2}),\quad \tilde{n}\to\infty\,.
\end{equation}
Therefore we have the estimate
\begin{equation}
\Big|\int_{X} u(\bx) e^{i\tilde{n}\tilde{\phi}(\bx)}d\bx\Big|  \le \cO(\tilde{n}^{-2}), \quad \tilde{n}\to\infty\,.
\end{equation}
When $\tilde{\phi}$ has degenerated isolated stationary points on $\partial X$, then there is an $1\le \alpha \le 2$ such that the boundary integral~\eqref{EQ:BD3} is bounded by $\cO(\tilde{n}^{-\alpha})$, then
\begin{equation}
\Big|\int_{X} u(\bx) e^{i\tilde{n}\tilde{\phi}(\bx)}d\bx\Big|  \le \cO(\tilde{n}^{-\alpha}), \quad \tilde{n}\to\infty\,.
\end{equation}
When $\tilde{\phi}$ has non-isolated stationary points, which means $\partial X$ coincides with part of the level set of $\tilde{\phi}$, then the boundary integral~\eqref{EQ:BD3} will be $\cO(\tilde{n}^{-1})$ instead. Hence in this case, we have
\begin{equation}
\Big|\int_{X} u(\bx) e^{i\tilde{n}\tilde{\phi}(\bx)}d\bx\Big|  \le \cO(\tilde{n}^{-1}), \quad \tilde{n}\to\infty\,.
\end{equation}  
Since $\hat{\phi}$ does not have stationary points in $X$,  same analysis can be applied to conclude that there is an $1\le \alpha \le 2$ such that
\begin{equation}
\Big|\int_{X} u(\bx) e^{i\hat{n}\hat{\phi}(\bx)}d\bx\Big|  \le \cO(\hat{n}^{-\alpha}) \le \cO(\tilde{n}^{-\alpha}), \quad \hat{n}\to\infty\,.
\end{equation} 
From the relation $|\cL_2| \le \tilde{n}^{-1}\cO(\cL_1)$ and $|\cL_3|\le \tilde{n}^{-1} \cO(\cL_1)$ as $\tilde{n}\to\infty$, we complete the proof.

\end{proof}
\begin{remark}
	When projection of $\by$ onto $xy$-plane is contained in the projection of $X$ onto $xy$-plane, the function approximation~\eqref{EQ:LEGEND} is not applicable anymore.
\end{remark}

The above lemma identifies a length scale ${n}^{-1}$ at which the kernel function $G_{n0}(\cdot, \by)$ decorrelates. From this decorrelating behavior, we follow the approach in ~\cite{EnZh-CPAM18} to show a lower bound for the dimension of a linear space in $L^2(X)$ that can approximate a discrete set of functions $G_{n0}(\bx,\by_m)$ to an $\eps$ error in root mean square sense. 

\begin{lemma}\label{lemmaPCA}
	Let $X$ and $Y$ be two disjoint convex compact domains in $\mathbb{R}^3$. Suppose $X$ and $Y$ have disjoint projections onto $xy$-plane and $z$-axis, then for any $\delta \in (0, 1)$, there are points $\by_m\in Y$, $m=1,2,\dots, M_{\delta}\sim n^{3-3\delta}$, such that the matrix $A = [a_{mk}]_{M_{\delta}\times M_{\delta}}$ with entry $a_{mk} = \cC(\by_m, \by_k)$ satisfies the following:
    let $\lambda_1 \ge \lambda_2\ge \dots\ge \lambda_{M_{\delta}}\ge 0$ be the eigenvalues of $A$ and $N_{\delta}^{\eps} = \min M$ such that $\sum_{m=M+1}^{M_{\delta}} \lambda_m \le \eps^2 \sum_{m=1}^{M_{\delta}} \lambda_{m}$. If the correlation function $\cC(\by_m, \by_k)$ satisfies
    \begin{equation}
    |\cC(\by_m, \by_k)|\le \cO((n|\by_m - \by_k|)^{-\alpha})\,,
    \end{equation}
    then
    \begin{equation}
    N_{\delta}^{\eps}\ge
    \begin{cases}
          \cO(n^{3-3\delta}), &\text{ if $\alpha\ge \frac{3}{2}$},\\
        \cO(n^{2\alpha}), &\text{ if $\alpha < \frac{3}{2}$}.
    \end{cases}
    \end{equation}
\end{lemma}
\begin{proof}
    Without loss of generality, we assume that $Y$ contains a unit cube. Then take $\by_m, m=1, 2, \ldots, M_{\delta}\sim n^{3-3\delta}$, the grid points of a uniform grid in $Y$ with grid size $h = n^{\delta-1}$, $\delta \in (0, 1)$. The matrix $A = [a_{mk}]_{M_{\delta}\times M_{\delta}}$, $a_{mk}=\cC(\by_m, \by_k)$, has the following properties
    \begin{equation}
    a_{mm} = 1,\quad  
    |a_{mk}| \le \cO((n |\by_m - \by_k|)^{-\alpha})\,,
    \end{equation}
    where $1\le\alpha \le 2$. Meanwhile, we have
    \begin{equation}
       \sum_{m=1}^{M_{\delta}} \lambda_m = \mathrm{tr}(A) = M_{\delta}\,.
    \end{equation}
    By the definition of $N^\eps$, that is, $N^{\eps} = \min M$ such that $\sum_{m=M+1}^{N_{\delta}} \lambda_m \le \eps^2 \sum_{m=1}^{N_{\delta}} = \eps^2 M_{\delta}$,  we obtain
    \begin{equation}
    \sum_{m=1}^{N_{\delta}^{\eps}} \lambda_m \ge (1-\eps^2) \sum_{m=1}^{M_{\delta}} \lambda_m = (1-\eps^2) M_{\delta} 
    \end{equation}
    and
    \begin{equation}
    \sum_{m=1}^{M_{\delta}} \lambda_m^2 > \sum_{m=1}^{N_{\delta}^{\eps}} \lambda_m^2 \ge N_{\delta}^{\eps}\left[ \frac{(1-\eps^2)M_{\delta}}{N_{\delta}^{\eps}}\right]^2 = \frac{[(1-\eps^2)M_{\delta}]^2}{N_{\delta}^{\eps}}\,.
    \end{equation}
    Therefore
    \begin{equation}\label{EQ:ESTIMATE N}
    N_{\delta}^{\eps} \ge \frac{[(1-\eps^2)M_{\delta}]^2}{\sum_{m=1}^{M_{\delta}} \lambda_m^2} = \frac{[(1-\eps^2)M_{\delta}]^2}{\text{tr}(A^TA)} = \frac{[(1-\eps^2)M_{\delta}]^2}{\sum_{m=1}^{M_{\delta}}\sum_{k=1}^{N_\delta} |a_{mk}|^2} \,.
    \end{equation}
    On the other hand, denote $\text{dist}(\by_m, \by_k)$ as the Hamming distance between $\by_m$ and $\by_k$, we can estimate the following summation in terms of $\text{dist}(\by_m, \by_k)$:
    \begin{equation}
    \begin{aligned}
    \sum_{m=1}^{M_{\delta}} |a_{mk}|^2 =1 + \sum_{\substack{j=1\\ \text{dist}(\by_m, \by_k) = jh}}^{n^{1-\delta}} |a_{mk}|^2\,. 
    \end{aligned}
    \end{equation}
    The first terms means the case $\by_m = \by_k$.
    The second summation will be grouped into box surfaces by the distances. Using the estimate $|a_{mk}|\le \cO((n|\by_m - \by_k|)^{-\alpha})$ and equivalence between Euclidean distance and Hamming distance, we obtain the bound
    \begin{equation}\label{EQ:LAYER}
    \sum_{\substack{j=1\\ \text{dist}(\by_m, \by_k) = jh}}^{O(n^{1-\delta})} |a_{mk}|^2 \le \sum_{j=1}^{n^{1-\delta}} \cO(j^{2}) \left(\frac{1}{n j h}\right)^{2\alpha} =
    \begin{cases}
    \cO(n^{-2\alpha\delta}) &\text{if $\alpha > \frac{3}{2}$},\\
     \cO(n^{-2\alpha\delta}\log n) &\text{if $\alpha = \frac{3}{2}$},\\
    \cO(n^{3-2\alpha -3\delta}) &\text{if $\alpha < \frac{3}{2}$}.
    \end{cases}
    \end{equation}
    The next step is to combine the $1$ and the estimate~\eqref{EQ:LAYER} for different choices of $\alpha$. This yields the following result.
    \begin{enumerate}
        \item When $\alpha\ge \frac{3}{2}$, we obtain
        \begin{equation}
        \sum_{m=1}^{M_{\delta}} |a_{mk}|^2 = \cO(1)\,.
        \end{equation}
        Using the relation~\eqref{EQ:ESTIMATE N}, we conclude that the lower bound of $N_{\delta}^{\eps}$ satisfies
        \begin{equation}
        N_{\delta}^{\eps} \ge 
        \cO(n^{3-3\delta}).
        \end{equation}
        \item When $\alpha<\frac{3}{2}$, we obtain
        \begin{equation}
        \sum_{m=1}^{M_{\delta}} |a_{mk}|^2 = \begin{cases}
        \cO(1),&\quad \delta > \frac{3-2\alpha}{3}\,, \\
        \cO(n^{3-2\alpha -3\delta}),&\quad \delta \le \frac{3-2\alpha}{3}\,.
        \end{cases}
        \end{equation}
        Therefore the lower bound of $N_{\delta}^{\eps}$ satisfies
        \begin{equation}
        N_{\delta}^{\eps} \ge \begin{cases}
        \cO(n^{3-3\delta}), &\quad \delta > \frac{3-2\alpha}{3}\,,\\
        \cO(n^{2\alpha}),&\quad \delta \le \frac{3-2\alpha}{3}\,.
        \end{cases}
        \end{equation}
        When $\delta = \frac{3-2\alpha}{2\alpha}$, the above lower bounds join at $\cO(n^{2\alpha})$.
    \end{enumerate}

\end{proof}

\begin{remark}
The above lemma is equivalent to the principal component analysis (PCA) of the set of unit vectors $\frac{G_{n0}(\cdot, \by_m)}{\|G_{n0}(\cdot, \by_m)\|_2}, m=1, 2, \ldots, M_{\delta}$ in $L^2(X)$.
The leading $N^{\eps}_{\delta}$ eigenvectors of $A$ form an orthonormal basis of the best linear space of dimension $N^{\eps}_{\delta}$ that approximates the set of functions $\frac{G_{n0}(\cdot, \by_m)}{\|G_{n0}(\cdot, \by_m)\|_2}$ in least square sense in $L^2(X)$. Since $\|G_{n0}(\cdot, \by)\|_2$ is uniformly bounded for $\by\in Y$, it is easy to see that  bounds of the same order hold for $G_{n0}(\cdot, \by_m), m=1, 2, \ldots, M_{\delta}$.
\end{remark}


With the above estimate in the discrete setting for $\by\in Y$, we can follow the technique in \cite{EnZh-CPAM18} to derive the following theorem for continuous case in $L^2(X\times Y)$ with a two-grid approach. We first use a grid (finer than the grid used in Lemma~\ref{lemmaPCA}) to approximate the integration over $Y$ by an integral of a piecewise constant function in $\by$ on the fine grid, which reduces the continuous case to a discrete setting. We then apply Lemma~\ref{lemmaPCA} to a coarse grid used in the lemma.
\begin{theorem}\label{THM:3D2}
Let $X$ and $Y$ be two disjoint convex compact domains in $\mathbb{R}^3$. Assume that $X$ and $Y$ have disjoint projections onto $xy$-plane and $z$-axis. For any $\eps > 0$, if there are functions $f_l(\bx)\in L^2(X)$ and $g_l(\by)\in L^2(Y)$, $l=1,\dots, N^{\eps}$, such that
\begin{equation}
\Big\|G_{n0}(\bx, \by) - \sum_{l=1}^{N^{\eps}} f_l(\bx) g_l(\by) \Big\|_{L^2(X\times Y)} \le \eps \|G_{n0}\|_{L^2(X\times Y)}\,.
\end{equation}
then as $n\to\infty$,
\begin{equation}
N^{\eps}\ge
\begin{cases}
\cO(n^{3-3\delta}), &\text{ if $\frac{3}{2}\le \alpha\le 2$},\\
\cO(n^{2\alpha}), &\text{ if $1\le \alpha <\frac{3}{2}$},
\end{cases}
\end{equation}
where $\alpha$ is defined in~\eqref{EQ: ALPHA} and $\delta\in(0,1)$ is an arbitrary number. 
\end{theorem}

\begin{remark}
    Let us point out that the case $|\cC(\by_1, \by_2)| \le  \cO(\tilde{n}^{-2})$ in Lemma~\ref{LEM:3D1} is generic for a general domain $X$ and two points $\by_1,\by_2$ in $Y$. Hence $\cO(n^{3-3\delta}), \forall \delta >0$ is the generic lower bound for the approximate separability $N^{\eps}$ for the 3D integral kernel $G_{n0}(\bx, \by)$.
\end{remark}

\begin{remark}
The kernel $G_{nn}$ (or $G_{n,-n}$) is related to the special spherical harmonics at $|m|= n$ given in the following form
\begin{equation}\label{EQ:YNN}
Y_{n,\pm n}(\theta, \varphi) = \frac{(\mp 1)^n}{2^n n!}\sqrt{\frac{(2n+1)!}{4\pi}}\sin^n \theta e^{\pm i n \varphi}\,.
\end{equation}
With this explicit oscillatory form in $\theta$ and $\varphi$, it is not hard to perform the previous analysis to get the lower bound for certain simple cases. For instance, when the convex compact domains $X$ and $Y$ have disjoint projections on both $z$-axis and $xy$-plane, with similar notations to those in Lemma~\ref{LEM:3D1}, we can show that the correlation function $\cC(\by_1, \by_2)$ involves integral
\begin{equation}
	\int_{X} e^{i n \phi} u(\bx) d\bx,\quad \phi = (\varphi_1 - \varphi_2) - i \log(\sin\theta_1\sin\theta_2)\,,
\end{equation}
which does not have any stationary point inside $X$.
The analysis of an arbitrary kernel $G_{nm}$ involves more sophisticated spherical harmonics, and is therefore much more complicated. 
\end{remark}

%
%
\subsection{Separability's upper bounds}
\label{SEC:UPPER}

We now establish the upper bound for $N^{\eps}$ for the approximate separability of the kernel functions ${G}_n(\bx,\by)$ in 2D and $G_{n0}(\bx,\by)$ in 3D (defined by~\eqref{EQ:K2} and~\eqref{EQ:K3} respectively) in terms of tolerance $\eps$ and $n$ under certain regularity assumptions. In particular, we use polynomials to construct separable approximations in $L^{\infty}(X\times Y)$. Since $X,Y$ are compact, $L^{\infty}$ is stronger than $L^2$. Hence the upper bound holds in  $L^2(X\times Y)$ In our analysis, we assume that ${\sigma}_t(\bx)\in C^{k+1}(\overline{\Omega}), k\ge 1$ is a real-valued function on $\Omega$. 

\subsubsection{The two-dimensional case}

\begin{theorem}\label{THM:LOWRANK}
	Let $X, Y$ be two disjoint compact sets in $\Omega\subset \bbR^2$ and $\bx_c \in X, \by_c \in Y$ be their centers respectively. The distance between the two centers is $|\bx_c - \by_c| =  \rho = \cO(1)$ and $\sup_{\bx\in X}|\bx_c - \bx|= \zeta$, $\sup_{\by\in Y}|\by_c - \by|= \eta$. Assume further that $\frac{(\zeta + \eta)}{\rho} < \frac{1}{2}$. Then for any $\eps >0$,  there exits $N^{\eps} \le\cO((n+\log\eps)^2\eps^{-4/(k+1)})$, and functions $f_l\in L^{\infty}(X)$, $g_l\in L^{\infty}(Y)$, $l=1,2,\dots,N^{\eps}$, such that
	\begin{equation}
	\Big\|\frac{E(\bx, \by)}{|\bx - \by|}e^{-in \arg(\bx - \by)} - \sum_{l=1}^{N^{\eps}} f_l(\bx) g_l(\by)\Big\|_{L^{\infty}(X\times Y)} \le \eps \,,
	\end{equation}
    where the constant in the upper bound for $N^{\eps}$ depends on $X$ and $Y$.
\end{theorem}
\begin{proof}
	Since $X$ and $Y$ are compact and disjoint, and ${\sigma}_t(\bx)\in C^{k+1}(\overline{\Omega})$, $h(\bx, \by):=
	\frac{E(\bx, \by)}{|\bx - \by|}$ is $ C^{k+1}(X\times Y)$. Without loss of generality, we assume that $X\subset \bigcup_{i=1}^{N_X} C_i$ and $Y\subset \bigcup_{j=1}^{N_Y} D_j$, where $C_i$ and $D_j$ are disjoint square cells of size $\ell$ and $N_X = \cO(\ell^{-d})$, $N_Y = \cO(\ell^{-d})$ with $d=2$. Let the centers of $C_i$ and $D_j$ be $\bx_i$ and $\by_j$ respectively. From Taylor expansion, locally for $\bx\in C_i$ and $\by\in D_j$, we have that
	\begin{equation}\label{EQ:EXPAND}
	h(\bx, \by) = \sum_{|\alpha|\le k, |\beta|\le k} \frac{D_{\bx}^{\alpha} D_{\by}^{\beta} h(\bx_i, \by_j)}{\alpha! \beta!} (\bx - \bx_i)^{\alpha}(\by - \by_j)^{\beta} +\cO(\ell^{k+1})\,.
	\end{equation}
We choose the cell size $\ell = \cO(\eps^{1/(k+1)})$ such that the remainder's magnitude in~\eqref{EQ:EXPAND} is strictly less than $\eps$. Then the following $|N_X||N_Y| = \cO(\eps^{-2d/(k+1)})$ function pairs 
\begin{equation}\label{EQ:POLY}
(\chi_{C_i} (\bx - \bx_i)^{\alpha}, \chi_{D_j}(\by - \by_j)^{\beta})\in L^{\infty}(X)\times L^{\infty}(Y),\quad  1\le i\le |N_X|,\; 1\le j \le |N_Y|,\; |\alpha|\le k, \; |\beta|\le k\,,
\end{equation}
$\chi_{K}$ being the characteristic function of set $K$, provide a piecewise polynomial separable approximation to $h(\bx, \by)$ within an error of $\eps$ in $L_{\infty}(X\times Y)$.  Next, let us estimate the separability of $e^{-in\arg(\bx - \by)}$. Since the function is only determined by the relative locations of $\bx$ and $\by$, we assume that the origin is at $\bx_c$ to obtain
\begin{equation}
\frac{|\bx|}{|\by|} = \frac{|\bx - \bx_c|}{|\by -\by_c +\by_c - \bx_c|} \le \frac{\zeta}{\rho - \eta} <  \frac{1}{2}\,.
\end{equation}
On the other hand, 
\begin{equation}~\label{EQ:EXP}
	e^{-i n \arg(\bx - \by)} = \left( \frac{|\bx|\cos\theta_1 - |\by|\cos\theta_2 - i(|\bx|\sin\theta_1 - |\by|\sin\theta_2)}{|\bx - \by|} \right)^n\,,
	\end{equation}
	where $\theta_1$ and $\theta_2$ are the polar angles for $\bx$ and $\by$ respectively. The numerator in~\eqref{EQ:EXP} is
	\begin{equation}\label{EQ:NUM}
\Big[	|\bx|\cos\theta_1 - |\by|\cos\theta_2 - i(|\bx|\sin\theta_1 - |\by|\sin\theta_2) \Big]^n =|\by|^n e^{-in\theta_2}\sum_{k=0}^n \binom{n}{k} (-1)^k|\bx|^k |\by|^{-k}  e^{-ik(\theta_1-\theta_2)}\,.
	\end{equation}
	The denominator of~\eqref{EQ:EXP} can be expressed by the generating function of the Gegenbauer polynomials~\cite{Rainville-Book71} 
	\begin{equation}\label{EQ:DEN}
	\frac{1}{|\bx - \by|^n} = \sum_{s=0}^{\infty} C_{s}^{n/2}(\cos(\theta_1 -\theta_2)) \frac{|\bx|^s}{|\by|^{s+n}} \,,
	\end{equation}
	which is convergent when $|\bx|< |\by|$. Moreover, from the estimate~\cite{Rainville-Book71},
	\begin{equation}
	|C_s^{n/2}(x)| \le C_s^{n/2}(1) = \frac{\Pi_{i=0}^{s-1}(n+i)}{s!},
	\end{equation}
	we have that, when $s > 2n$,
	\begin{equation}
\frac{\Pi_{i=0}^{s}(n+i)}{(s+1)!} \Big\slash 	\frac{\Pi_{i=0}^{s-1}(n+i)}{s!} = \frac{s+n}{s+1} < \frac{3}{2}\,.
	\end{equation}
	Since $\frac{|\bx|}{|\by|}<\frac{1}{2}$, a truncation of $N = 2n + \cO(|\log\eps|)$ terms is needed in~\eqref{EQ:DEN} to achieve an error less than $\eps$. 
	
	The function $C_s^{n/2}(x)$ is a polynomial of order $s$, therefore
	\begin{equation}\label{EQ:GEGEN}
	\begin{aligned}
	C_s^{n/2}(\cos(\theta_1 - \theta_2)) &= C_s^{n/2}\left(\frac{e^{i(\theta_1 -\theta_2)} + e^{-i(\theta_1- \theta_2)}}{2}\right) \\
	&= \sum_{t=0}^s c_{st}^{n/2} \sum_{l=0}^t \frac{1}{2^t}\binom{t}{l} e^{i(t-2l)(\theta_1 - \theta_2)}\,,
	\end{aligned}
	\end{equation}
	where $c_{st}^{n/2}$ is the coefficient of $x^t$ of the Gegenbauer polynomial $C_{s}^{n/2}(x)$. Combine ~\eqref{EQ:NUM},~\eqref{EQ:DEN} and~\eqref{EQ:GEGEN}, 
	\begin{equation}
	\Big| e^{-in\arg(\bx - \by)}- e^{-in\theta_2} \sum_{k=0}^n \sum_{s=0}^N \sum_{t=0}^s \sum_{l=0}^t \binom{n}{k} \binom{t}{l}(-1)^k \frac{1}{2^t}c_{st}^{n/2} |\bx|^{k+s}  |\by|^{-(k+s)}    e^{i(t-2l-k)(\theta_1 - \theta_2)} \Big|\le \eps\,.
	\end{equation} 
	It can be easily seen that $0\le k+s\le 2N$ and $-2N \le t - 2l - k \le 2N$. Therefore we can choose the following functions
	\begin{equation}
	\begin{aligned}
	p_{k,l}(\bx) &= |\bx|^{k} e^{i l \theta_1}, &\quad 0\le k\le 2N, -2N\le l\le 2N\,, \\
	q_{k,l}(\by) &= |\by|^{-k} e^{-i(l+n) \theta_2},&\quad 0\le k \le 2N, -2N \le l \le 2N\,,
	\end{aligned}
	\end{equation}
	and some constants $\gamma_{k,l}$ such that
	\begin{equation}
	\big| e^{-in\arg (\bx - \by)} - \sum_{k=0}^{2N} \sum_{l=-2N}^{2N}\gamma_{k,l} p_{k,l}(\bx) q_{k,l}(\by)  \big| \le \eps\,.
	\end{equation}
	It is now clear that the tensor product of $\{p_{kl}, q_{kl}\}$ with the functions in~\eqref{EQ:POLY} is a choice for the functions $\{f_l, g_l\}$ in the theorem.
\end{proof}
\begin{remark}
When the coefficient $\sigma_t$ is analytic in $\Omega$, one can replace $\eps^{-2d/(k+1)}$ by $(\log\eps)^{2d}$. In particular, when $n$ is small, the kernels $G_n$ with homogeneous and analytic total absorption coefficient $\sigma_t$ admit a separability with $\cO(|\log \eps|^6)$ terms. Such low-rank structure has been computationally observed in a previous work~\cite{ReZhZh-arXiv19}.
\end{remark}

\subsubsection{The three-dimensional case}

We first show an asymptotic upper bound for the separability of $G_{n0}(\bx,\by)$ as $n\rightarrow \infty$.
\begin{theorem}
    Let $X$ and $Y$ two compact domains embedded in $\bbR^3$. Suppose that $X$ and $Y$ have disjoint projections onto the $xy$-plane. For any $\eps > 0$ and $\delta > 0$, there exist $N^{\eps} \le \cO(n^{3+\delta})$ and functions $f_l(\bx)\in L^{\infty}(X), g_l(\by) \in L^{\infty}(Y)$, $l=1,2,\dots, N^{\eps}$ such that 
    \begin{equation}
    \left\|G_{n0}(\bx, \by) - \sum_{l=1}^{N^{\eps}} f_l(\bx) g_l(\by)\right \|_{L^{\infty}(X\times Y)} \le \eps\,,
    \end{equation}
   for sufficiently large $n$, where the constant in the upper bound for $N^{\eps}$ depends on $X$ and $Y$.
\end{theorem}
\begin{proof}
    Without loss of generality, we assume $Y$ is contained in a unit cube. Let $\by_m$, $m=1,2,\dots, N^h = n^{3(1+\delta/3)}$ be the grid points of a uniform Cartesian grid in $Y$ with a grid size $h = n^{-1 - \delta/3}$. We denote the linear subspace $S_X = \text{span}\{ G_{n0}(\bx, \by_m)  \}_{m=1}^{N^h}\subset L^{\infty}(X)$. Then we only have to show that 
    \begin{equation}
    \left\| G_{n0}(\bx, \by) - P_{S_X} G_{n0}(\bx, \by)  \right\|_{L^{\infty}(X)} \le \eps\,,
    \end{equation}
    $P_{S_X}$ being the projection onto $S_X$, for sufficiently large $n$. Since $X$ and $Y$ are disjoint and their projections onto the $xy$-plane are disjoint as well, the polar angle $\theta$ of $\bx - \by$ is away from $0$ and $\pi$. Hence there exists a constant $c'$ that $\theta \in [c', \pi - c']$ for any $\bx\in X$ and $\by\in Y$. From the asymptotic expansion~\eqref{EQ:LEGEND},
    \begin{equation}
    \begin{aligned}
    \nabla_{\by} G_{n0}(\bx, \by) &= \sqrt{\frac{\hat{n}}{2\pi}} \left( P_n(\cos \theta)\nabla_{\by} \frac{E(\bx, \by)}{|\bx - \by|^2} + \frac{E(\bx, \by)}{|\bx - \by|^2} \nabla_{\by} P_n(\cos\theta) \right) \\ 
    &\le \cO(1) + \cO(n) = \cO(n)
    \end{aligned}
    \end{equation}
    and $|\nabla^2_{\by} G_{n0}(\bx, \by)| \le \cO(n^2)$, where the bounds are uniform in $\bx, \by$, and the constants depend on $c'$ and the distance between $\bx, \by$. Then follow the proof of Theorem 3.2 in~\cite{EnZh-CPAM18}, given any non-grid point $\by \in Y$, $G_{n0}(\bx, \by)$ can be approximated by a linear interpolation of $G_{n0}(\bx, \by_m)$ at neighboring grid points. Suppose $\by_{m_1},\dots, \by_{m_{d+1}}$ form the $d$-simplex containing $\by$, then the barycentric coordinates $\lambda_{j} \ge 0$ satisfies
    \begin{equation}
    \by = \sum_{j=1}^{d+1} \lambda_j \by_{m_j},\quad \sum_{j=1}^{d+1} \lambda_j = 1\,.
    \end{equation} 
    Therefore
    \begin{equation}
    \left| G_{n0}(\bx, \by) - \sum_{j=1}^{d+1} \lambda_j G_{n0}(\bx, \by_{m_j})\right| \le \cO\left( h^2 \sup_{\by\in Y}\|\nabla^2_{\by} G_{n0}(\bx, \by)\|\right) = \cO(n^2 h^2) 
    = \cO(n^{-2\delta/3})\,.
    \end{equation}
    This completes the proof.
\end{proof}

\begin{remark}
When the projections of $X$ and $Y$ onto $xy$-plane are overlapped, $|\nabla_{\by} G_{n0}(\bx, \by)|$ will be bounded by $\cO(n^{3/2})$ instead. This means that the upper bound could be larger than $\cO(n^{3+\delta})$. In fact, we can construct a $\cO(n^{4+\delta})$-term separable approximation for the general case, see Theorem~\ref{THM:3D3}.
\end{remark}

To prove the main result of this subsection, we need the following lemma.
\begin{lemma}\label{LEM:LEG}
	The Legendre polynomial $P_n(z)$ for $|z|\le 1$ has the following bound,
	\begin{equation}
	\sup_{|z| = 1}|P_n(z)| = |P_n(\pm i) |
	\end{equation}
	and $\sup_{|z| = 1}|P_n(z)|\le 3^n$.
\end{lemma}
\begin{proof}
	First, we can write the Legendre polynomial in the explicit form:
	\begin{equation}
	P_n(z) = \frac{1}{2^n}\sum_{k=0}^{[\frac{n}{2}]} (-1)^k \binom{n}{k}\binom{2n-2k}{n} z^{n-2k}\,.
	\end{equation}
	We then use the fact $i^{n-2k}(-1)^k \equiv i^n$ to obtain the following bound:
	\begin{equation}
	\sup_{|z|\le 1} |P_n(z) | \le \frac{1}{2^n}\sum_{k=0}^{[\frac{n}{2}]} \binom{n}{k}\binom{2n-2k}{n} = |P_n(\pm i)|\,.
	\end{equation}
	The estimate is based on the Schl\"{a}fli's integral representation~\cite{Schlafli-GMA56},
	\begin{equation}
	P_n(z) = \frac{1}{2\pi i} \oint_{C} \frac{(w^2 - 1)^n}{2^n (w-z)^{n+1}} dw
	\end{equation}
	with $C$ being any simple counter-clockwise loop around $z$. By taking the loop as a circle centered at $z$ with radius $\sqrt{|z^2 - 1|}$~\cite{ByFu-Book12}, we have
	\begin{equation}
	P_n(z) = \frac{1}{\pi} \int_0^{\pi} (z + \sqrt{z^2 - 1}\cos t)^n dt\,,
	\end{equation}
	when $z = \pm i$, $|z + \sqrt{z^2 - 1}\cos t| \le |1 + \sqrt{2}\cos t| < 3$, therefore $
	|P_n(\pm i)| \le 3^n$.
\end{proof}

Here is the main result on the upper bound in the three-dimensional case.
\begin{theorem}\label{THM:3D3}
	Let $X, Y$ be two disjoint compact sets in $\Omega \subset \bbR^3$ and $\bx_c \in X$ and $\by_c \in Y$ be their centers respectively.  The distance between the two centers is $|\bx_c - \by_c| =  \rho =\cO(1)$ and $\sup_{\bx\in X}|\bx_c - \bx|= \zeta$, $\sup_{\by\in Y}|\by_c - \by|= \eta$. Assume $n\ge 1$ and $\frac{(\zeta + \eta)}{\rho} < \frac{1}{2}$, then for any $\eps >0$, there exists $N^{\eps} \le \cO((n+\log\eps)^4\eps^{-6/(k+1)})$ and functions $f_l\in L^{\infty}(X), g_l\in L^{\infty}(Y)$, $l=1,2,\dots,N^{\eps}$ such that
	\begin{equation}
	\Big|\frac{E(\bx, \by)}{|\bx - \by|^2} Y_{n0}\left(\frac{\bx - \by}{|\bx - \by|}\right) - \sum_{l=1}^{N^{\eps}} f_l(\bx) g_l(\by)\Big| \le \eps\,,
	\end{equation}
	where the constant in $N^{\eps}$ depends on $X$ and $Y$.
\end{theorem}
\begin{proof}
	From the first part of the proof in Theorem~\ref{THM:LOWRANK}, we can conclude that the separability of $\frac{E(\bx, \by)}{|\bx - \by|^2}$ is at most $\cO(\eps^{-2d/(k+1)})$ with $d = 3$. Now we only discuss the separability of the function $Y_{n0}\left(\frac{\bx - \by}{|\bx - \by|}\right)$. Again due to the fact that separability does not depend on the choice of the origin, we can select the origin at $\bx_c$ to have $\frac{|\bx|}{|\by|} < \frac{1}{2}$. Using the spherical coordinate $\bx = (|\bx|, \theta_1, \phi_1)$ and $\by = (|\by|, \theta_2, \phi_2)$, where $\theta_i$ is the polar angle and $\phi_i$ the azimuth angle respectively, we have
	\begin{equation}\label{EQ:YNO EXP}
	\begin{aligned}
Y_{n0}\left(\frac{\bx - \by}{|\bx - \by|}\right) &= \sqrt{\frac{2n+1}{4\pi}} P_n\left(\frac{|\bx| \cos\theta_1 - |\by|\cos\theta_2}{|\bx - \by|}\right) \\
&= \sqrt{\frac{2n+1}{4\pi}} \sum_{k=0}^n c_{nk}^{1/2} \left(|\bx| \cos\theta_1 - |\by|\cos\theta_2\right)^k |\bx - \by|^{-k} \\
&= \sqrt{\frac{2n+1}{4\pi}} \sum_{k = 0}^n \sum_{j=0}^k c_{nk}^{1/2}\binom{k}{j}(-1)^j |\bx|^j |\by|^{k-j} \cos^j\theta_1\cos^{k-j}\theta_2 \cdot |\bx - \by|^{-k}\,,
	\end{aligned}
	\end{equation}
	where $c_{nk}^{1/2}$ is the coefficient of $x^k$ of the Legendre polynomial $C_n^{1/2}(x)$. On the other hand, we can expand $|\bx - \by|^{-k}$ with the generating function of the Gegenbauer polynomials to have
	\begin{equation}\label{EQ:SERIES}
	\frac{1}{|\bx - \by|^k} = \sum_{s = 0}^{\infty} C_s^{k/2}(\Delta) \frac{|\bx|^s}{|\by|^{s + k}}\quad \text{with }\Delta = \cos\theta_1 \cos\theta_2 + \sin\theta_1\sin\theta_2 \cos(\phi_1 - \phi_2)
	\end{equation}
 with $C_s^{k/2}$ the same as in the 2D case. Following the same argument in Theorem~\ref{THM:LOWRANK}, for each $|\bx - \by|^{-k}$, the truncated series from~\eqref{EQ:SERIES} with  $$N_k = 2k +\cO\left(\log\left(\frac{\eps}{\sqrt{n}3^n |\bx - \by|^k}\right)\right)\le \cO(n)+\cO(\log n) + \cO(|\log\eps|)$$ terms has an approximation error less than $\sqrt{\dfrac{4\pi}{2n+1}}\dfrac{\eps}{3^n |\bx  - \by|^k}$, where the constants depend on the distance between $X$ and $Y$. Therefore we can put the above truncated series into~\eqref{EQ:YNO EXP}, leading to an approximation error at most 
	\begin{equation}
	\begin{aligned}
	&\quad \Big|Y_{n0}\left(\frac{\bx - \by}{|\bx - \by|}\right) - \sqrt{\frac{2n+1}{4\pi}} \sum_{k = 0}^n \sum_{j=0}^k \sum_{s=0}^{N_k} c_{nk}^{1/2}\binom{k}{j}(-1)^j C_s^{k/2}(\Delta) |\bx|^{j+s} |\by|^{-(j+s)} \cos^j\theta_1\cos^{k-j}\theta_2    \Big| \\&\le
	\sqrt{\frac{2n+1}{4\pi}} \sum_{k=0}^n \Big| c_{nk}^{1/2} \left(|\bx| \cos\theta_1 - |\by|\cos\theta_2\right)^k \Big|\sqrt{\frac{4\pi}{2n+1}} \frac{\eps}{3^n |\bx - \by|^k} \\&
	= \frac{\eps}{3^n}\sum_{k=0}^n \Big| c_{nk}^{1/2} \left(\frac{|\bx| \cos\theta_1 - |\by|\cos\theta_2}{|\bx - \by|}\right)^k \Big| \le \frac{\eps}{3^n}\sum_{k=0}^n \big| c_{nk}^{1/2}\big| \le \eps\,.
	\end{aligned}
	\end{equation}
	The last inequality follows from Lemma~\ref{LEM:LEG}. We now continue to expand $C_s^{k/2}(\Delta)$ into 
	\begin{equation}
    \begin{aligned}
	C_s^{k/2}(\Delta) &= \sum_{t=0}^s c_{st}^{k/2} (\cos\theta_1 \cos\theta_2 + \sin\theta_1\sin\theta_2 \cos(\phi_1 - \phi_2))^t \\
    &= \sum_{t=0}^s c_{st}^{k/2} \sum_{l_1 + l_2 \le t} (\cos\theta_1 \cos\theta_2 )^{l_1} (\sin\theta_1\sin\theta_2)^{t-l_1} (\cos\phi_1\cos\phi_2)^{l_2} (\sin\phi_1\sin\phi_2)^{t-l_1 - l_2}\,.
    \end{aligned}
	\end{equation}
	Therefore the truncated expansion can be written in the following form:
	\begin{equation}
\sqrt{\frac{2n+1}{4\pi}} \sum_{k = 0}^n \sum_{j=0}^k \sum_{s=0}^{N_k}\sum_{t=0}^s \sum_{l_1 +l_2\le t} c_{nk}^{1/2}c_{st}^{k/2} \binom{k}{j}(-1)^j p_{k,j,s,t,l_1, l_2}(|\bx|, \theta_1, \phi_1) q_{k,j,s,t,l_1, l_2}(|\by|, \theta_2, \phi_2) 
	\end{equation}
	where $p_{k,j,s,t,l_1, l_2}$ and $q_{k,j,s,t,l_1,l_2}$ are given as:
	\begin{equation}
	\begin{aligned}
	p_{k,j,s,t,l_1, l_2} &= |\bx|^{j+s}  \cos^{j+l_1}\theta_1  \sin^{t-l_1} \theta_1 \cos^{l_2}\phi_1\sin^{t-l_1 - l_2}\phi_1\,, \\
	q_{k,j,s,t,l_1, l_2} &= |\by|^{-(j+s)} \cos^{k-j +l_1}\theta_2  	\sin^{t-l_1} \theta_2 \cos^{l_2}\phi_2\sin^{t-l_1 - l_2}\phi_2 \,.
	\end{aligned}
	\end{equation}
	The theorem is then proved with the observation that the set $\{p_{k,j,s,t,l_1, l_2}\}$ contains only $\cO(N_n^4)$ different functions of the form
	\begin{equation}
	|\bx|^{m} \cos^a \theta_1 \sin^b \theta_1 \cos^c \phi_1 \sin^d \phi_1 
	\end{equation}
	subject to the constraints
	\begin{equation}
	0 \le m \le 2N_n,\quad a + b \le m ,\quad c +d = b.
	\end{equation}
\end{proof}

\section{Concluding remarks}
\label{SEC:CON}

We studied in this work an integral formulation of the radiative transfer equation (RTE) with a generic anisotropic scattering phase function $p(\bv, \bv')$ that depends only on the product $\bv\cdot\bv'$. Unlike in the case of isotropic scattering where we can derive a single integral equation for the zeroth moment of RTE solution that is completely decoupled from its higher order moments, the integral formulation in the anisotropic case involves a system of integral equations that couples all angular moments of the RTE solution. We studied approximate separability, that is, separable approximation with certain accuracy tolerance, of the integral kernels of this coupled system of integral equations. More precisely, we developed asymptotic lower and upper bounds on the separability in both two- and three-dimensional physical space; see Theorem~\ref{thm:sep} and Theorem~\ref{THM:LOWRANK} respectively for lower and upper bounds in the two-dimensional case, and Theorem~\ref{THM:3D2} and Theorem~\ref{THM:3D3} respectively for lower and upper bounds in the three-dimensional case. A general observation is that, the separability indicator, that is the number of terms needed for the separable approximation, grows very fast (often at some power rate) with respect to the frequency in the angular space, but much slower with respect to the accuracy tolerance.

Integral formulations play important roles in developing fast algorithms for partial differential equations. Separability properties of the integral kernels decides whether or not (hierarchical) low-rank approximations exist for the integral operator. Low-rank approximations are often the foundation of fast computational algorithms for integral equations. In our case, the growth rates of the approximate separability of the integral kernels for the anisotropic radiative transfer equation provide some insight on the computational cost of the integral formulation of the RTE: when the scattering phase function $p(\bv\cdot \bv')$ is very anisotropic, we need a large number of terms in the approximate scattering phase function $p_M$ to have an accurate approximation. However, the corresponding integral kernels for large $M$ are very ``un-separable''. Therefore, the matrices corresponding to those kernels can not be compressed much, and thus require high computational cost to be multiplied to a given vector. This results in a high overall computational cost for a problem with large $M$. On the other hand, if the scattering phase function can be approximated accurately with only a small number of terms in $p_M$, for instance in the case of isotropic scattering as discussed in~\cite{ReZhZh-arXiv19}, one can have quite efficient compression for the integral kernels involved. Fast computational algorithms can be developed in this case.

\section*{Acknowledgment}
This work is partially supported by the National Science Foundation through grants DMS-1620473, DMS-1821010 and DMS-1913309.

\bibliographystyle{siam}
\bibliography{Separability-Bib}

\begin{thebibliography}{10}

\bibitem{Agoshkov-Book12}
{\sc V.~Agoshkov}, {\em Boundary Value Problems for Transport Equations},
  Springer Science \& Business Media, 2012.

\bibitem{Arridge-IP99}
{\sc S.~R. Arridge}, {\em Optical tomography in medical imaging}, Inverse
  Problems, 15 (1999), pp.~R41--R93.

\bibitem{Bal-IP09}
{\sc G.~Bal}, {\em Inverse transport theory and applications}, Inverse
  Problems, 25 (2009).
\newblock 053001.

\bibitem{BaCaLiRe-IP07}
{\sc G.~Bal, L.~Carin, D.~Liu, and K.~Ren}, {\em Experimental validation of a
  transport-based imaging method in highly scattering environments}, Inverse
  Problems, 23 (2007), pp.~2527--2539.

\bibitem{BaRe-SIAM08}
{\sc G.~Bal and K.~Ren}, {\em Transport-based imaging in random media}, SIAM J.
  Appl. Math., 68 (2008), pp.~1738--1762.

\bibitem{Bebendorf-MC05}
{\sc M.~Bebendorf}, {\em Efficient inversion of the galerkin matrix of general
  second-order elliptic operators with nonsmooth coefficients}, Math. Comp., 74
  (2005), pp.~1179--1199.

\bibitem{BeHa-NM03}
{\sc M.~Bebendorf and W.~Hackbusch}, {\em Existence of $\mathcal{H}$-matrix
  approximants to the inverse fe-matrix of elliptic operator with
  $l^{\infty}$-coefficients}, Numer. Math., 95 (2003), pp.~1--28.

\bibitem{BoGa-PRE16}
{\sc L.~Borcea and J.~Garnier}, {\em Derivation of a one-way radiative transfer
  equation in random media}, Phys. Rev. E, 93 (2016), p.~022115.

\bibitem{Borm-NM10}
{\sc S.~B\"{o}rm}, {\em Approximation of solution operators of elliptic partial
  differential equations by $\mathcal{H}$- and $\mathcal{H}^2$- matrices},
  Numer. Math., 115 (2010), pp.~165--193.

\bibitem{BrZhZh-MMS19}
{\sc J.~Bryson, H.~Zhao, and Y.~Zhong}, {\em Intrinsic complexity and scaling
  laws: From random fields to random vectors}, Multiscale Model. Simul., 17
  (2019), pp.~460--481.

\bibitem{ByFu-Book12}
{\sc F.~W. Byron and R.~W. Fuller}, {\em Mathematics of Classical and Quantum
  Physics}, Courier Corporation, 2012.

\bibitem{CaDeYi-MMS09}
{\sc E.~Candes, L.~Demanet, and L.~Ying}, {\em A fast butterfly algorithm for
  the computation of {Fourier} integral operators}, Multiscale Model. Simul., 7
  (2009), pp.~1727--1750.

\bibitem{CeBaBeAi-TTSP99}
{\sc C.~Cecchi-Pestellini, L.~Barletti, A.~Belleni-Morante, and S.~Aiello},
  {\em Radiative transfer in the stochastic interstellar medium}, Trans. Theor.
  Stat. Phys., 28 (1999), pp.~199--228.

\bibitem{Chandrasekhar-Book60}
{\sc S.~Chandrasekhar}, {\em Radiative {Transfer}}, Dover, New York, 1960.

\bibitem{DaLi-Book93-6}
{\sc R.~Dautray and J.-L. Lions}, {\em Mathematical {Analysis} and {Numerical}
  {Methods} for {Science} and {Technology}, {Vol VI}}, Springer-Verlag, Berlin,
  1993.

\bibitem{Davis-Book75}
{\sc P.~J. Davis}, {\em Interpolation and Approximation}, Courier Corporation,
  1975.

\bibitem{DeThUr-JCP12}
{\sc J.~D. Densmore, K.~G. Thompson, and T.~J. Urbatsch}, {\em A hybrid
  transport-diffusion monte carlo method for frequency-dependent
  radiative-transfer simulations}, J. Comput. Phys., 231 (2012),
  pp.~6924--6934.

\bibitem{DiRe-JCP14}
{\sc T.~Ding and K.~Ren}, {\em Inverse transport calculations in optical
  imaging with subspace optimization algorithms}, J. Comput. Phys., 273 (2014),
  pp.~212--226.

\bibitem{EnZh-CPAM18}
{\sc B.~Engquist and H.~Zhao}, {\em Approximate separability of the {Green's}
  function of the {Helmholtz} equation in the high frequency limit}, Comm. Pure
  Appl. Math., 71 (2018), pp.~2220--2274.

\bibitem{FaAnYi-JCP19}
{\sc Y.~Fan, J.~An, and L.~Ying}, {\em Fast algorithms for integral
  formulations of steady-state radiative transfer equation}, J. Comput. Phys.,
  380 (2019), pp.~191--211.

\bibitem{GaZh-TTSP09}
{\sc H.~Gao and H.~Zhao}, {\em A fast forward solver of radiative transfer
  equation}, Trans. Theor. Stat. Phys., 38 (2009), pp.~149--192.

\bibitem{GaZh-OE10}
{\sc H.~Gao and H.~Zhao}, {\em Multilevel bioluminescence tomography based on
  radiative transfer equation part 1: l1 regularization}, Optics Express, 18
  (2010), pp.~1854--1871.

\bibitem{GoYa-SIAM16}
{\sc F.~G{\"o}lgeleyen and M.~Yamamoto}, {\em Stability for some inverse
  problems for transport equations}, SIAM J. Math. Anal., 48 (2016),
  pp.~2319--2344.

\bibitem{Greengard-Thesis88}
{\sc L.~Greengard}, {\em The Rapid Evaluation of Potential Fields in Particle
  Systems}, MIT Press, Cambridge, MA, 1988.

\bibitem{GrRo-JCP87}
{\sc L.~Greengard and V.~Rokhlin}, {\em A fast algorithm for particle
  simulations}, J. Comput. Phys., 73 (1987), pp.~325--348.

\bibitem{Hackbusch-Computing99}
{\sc W.~Hackbusch}, {\em A sparse matrix arithmetic based on
  $\mathcal{H}$-matrices. i. introduction to $\mathcal{H}$-matrices.},
  Computing, 62 (1999), pp.~89--108.

\bibitem{Hackbusch-Book17}
\leavevmode\vrule height 2pt depth -1.6pt width 23pt, {\em Elliptic
  Differential Equations: Theory and Numerical Treatment}, Springer, 2017.

\bibitem{HeGr-AJ41}
{\sc L.~G. Henyey and J.~L. Greenstein}, {\em Diffuse radiation in the galaxy},
  The Astrophysical Journal, 93 (1941), pp.~70--83.

\bibitem{HoYi-CPAM16}
{\sc K.~Ho and L.~Ying}, {\em Hierarchical interpolative factorization for
  elliptic operators: differential equations}, Comm. Pure Appl. Math., 69
  (2016), pp.~1415--1451.

\bibitem{KiMo-IP06}
{\sc A.~D. Kim and M.~Moscoso}, {\em Radiative transport theory for optical
  molecular imaging}, Inverse Problems, 22 (2006), pp.~23--42.

\bibitem{Kolmogorov-AM36}
{\sc A.~Kolmogorov}, {\em \"{U}ber die beste ann\"{a}herung von funkionen einer
  funktionklasse}, Ann. Math., 37 (1936), pp.~107--111.

\bibitem{LaLiUh-SIAM18}
{\sc R.-Y. Lai, Q.~Li, and G.~Uhlmann}, {\em Inverse problems for the
  stationary transport equation in the diffusion scaling}, arXiv:1808.02071,
  (2018).

\bibitem{Larsen-NSE88}
{\sc E.~W. Larsen}, {\em Neutronics methods for thermal radiative transfer},
  Nuclear Science and Engineering, 100 (1988), pp.~255--259.

\bibitem{LeMi-Book93}
{\sc E.~E. Lewis and W.~F. Miller}, {\em Computational Methods of Neutron
  Transport}, American Nuclear Society, La Grange Park, IL, 1993.

\bibitem{LiSu-arXiv19}
{\sc Q.~Li and W.~Sun}, {\em Applications of kinetic tools to inverse transport
  problems}, arXiv:1908.00094,  (2019).

\bibitem{LuQiBu-JCP14}
{\sc S.~Luo, J.~Qian, and R.~Burridge}, {\em Fast huygens sweeping methods for
  helmholtz equations in inhomogeneous media in the high frequency regime}, J.
  Comput. Phys., 270 (2014), pp.~378--401.

\bibitem{MaRe-CMS14}
{\sc A.~V. Mamonov and K.~Ren}, {\em Quantitative photoacoustic imaging in
  radiative transport regime}, Comm. Math. Sci., 12 (2014), pp.~201--234.

\bibitem{MiBo-IEEE96}
{\sc E.~Michielssen and A.~Boag}, {\em A multilevel matrix decomposition
  algorithm for analyzing scattering from large structures}, IEEE Trans.
  Antennas and Propagation, 44 (1996), pp.~1086--1093.

\bibitem{Mokhtar-Book97}
{\sc M.~Mokhtar-Kharroubi}, {\em Mathematical Topics in Neutron Transport
  Theory: New Aspects}, vol.~46, World Scientific, 1997.

\bibitem{Rainville-Book71}
{\sc E.~D. Rainville}, {\em Special Functions}, Chelsea, 1971.

\bibitem{Ren-CiCP10}
{\sc K.~Ren}, {\em Recent developments in numerical techniques for
  transport-based medical imaging methods}, Commun. Comput. Phys., 8 (2010),
  pp.~1--50.

\bibitem{ReBaHi-SIAM06}
{\sc K.~Ren, G.~Bal, and A.~H. Hielscher}, {\em Frequency domain optical
  tomography based on the equation of radiative transfer}, SIAM J. Sci.
  Comput., 28 (2006), pp.~1463--1489.

\bibitem{ReZhZh-arXiv19}
{\sc K.~Ren, R.~Zhang, and Y.~Zhong}, {\em A fast algorithm for radiative
  transport in isotropic media}, arXiv:1610.00835,  (2019).

\bibitem{SaTaCoAr-IP13}
{\sc T.~Saratoon, T.~Tarvainen, B.~T. Cox, and S.~R. Arridge}, {\em A
  gradient-based method for quantitative photoacoustic tomography using the
  radiative transfer equation}, Inverse Problems, 29 (2013).
\newblock 075006.

\bibitem{Schlafli-GMA56}
{\sc L.~Schl{\"a}fli}, {\em {\"U}ber die zwei heineschen kugelfunktionen mit
  beliebigem parameter und ihre ausnahmslose darstellung durch bestimmte
  integrale}, in Gesammelte Mathematische Abhandlungen, Springer, 1956,
  pp.~317--392.

\bibitem{SpKuCh-JQSRT01}
{\sc R.~J.~D. Spurr, T.~P. Kurosu, and K.~V. Chance}, {\em A linearized
  discrete ordinate radiative transfer model for atmospheric remote-sensing
  retrieval}, J. Quant. Spectrosc. Radiat. Transfer, 68 (2001), pp.~689--735.

\bibitem{Szeg-Book39}
{\sc G.~Szeg}, {\em Orthogonal polynomials}, American Mathematical Society,
  1939.

\bibitem{Tamasan-IP02}
{\sc A.~Tamasan}, {\em An inverse boundary value problem in two-dimensional
  transport}, Inverse Problems, 18 (2002), pp.~209--219.

\bibitem{TuFrDuKl-JCP04}
{\sc R.~Turpault, M.~Frank, B.~Dubroca, and A.~Klar}, {\em Multigroup half
  space moment approximations to the radiative heat transfer equations}, J.
  Comput. Phys., 198 (2004), pp.~363--371.

\bibitem{Wang-AIHP99}
{\sc J.-N. Wang}, {\em Stability estimates of an inverse problem for the
  stationary transport equation}, Ann. Inst. Henri Poincar\'e, 70 (1999),
  pp.~473--495.

\bibitem{Wigner-Book59}
{\sc E.~Wigner}, {\em Group Theory and Its Application to the Quantum Mechanics
  of Atomic Spectra}, Academic Press, 1959.

\bibitem{Wiscombe-JAS77}
{\sc W.~Wiscombe}, {\em The delta--m method: Rapid yet accurate radiative flux
  calculations for strongly asymmetric phase functions}, J. Atmos. Sci., 34
  (1977), pp.~1408--1422.

\bibitem{ZhZh-SIAM18}
{\sc H.~Zhao and Y.~Zhong}, {\em Instability of an inverse problem for the
  stationary radiative transport near the diffusion limit}, arXiv:1809.01790,
  (2018).

\end{thebibliography}

\end{document}